\documentclass[reqno]{amsart}

\usepackage[usenames,dvipsnames]{pstricks}
\usepackage{epsfig}
\allowdisplaybreaks
\usepackage{color}
\date{\today}

\usepackage[latin1]{inputenc}
\usepackage[T1]{fontenc}
\usepackage{amsmath, amsthm}
\usepackage[english]{babel}
\usepackage{amssymb}
\usepackage{latexsym}
\usepackage{graphicx}
\usepackage[all]{xy}

\newtheorem{theorem}{Theorem}[section]
\newtheorem{lemma}[theorem]{Lemma}
\newtheorem{proposition}[theorem]{Proposition}
\newtheorem{corollary}[theorem]{Corollary}
\theoremstyle{definition}
\newtheorem{definition}[theorem]{Definition}
\newtheorem{example}[theorem]{Example}

\theoremstyle{remark}
\newtheorem{remark}[theorem]{Remark}

\def\<{\langle}
\def\>{\rangle}
\def\a{\alpha}

\def\k{\Bbbk}
\def\ci{\circ}
\def\c{\cdot}

\def\D{\Delta}

\def\i{\iota}

\def\r{\rho}

\def\o{\otimes}

\def\t{\tau}

\def\v{\varepsilon}

\def\<{\langle}
\def\>{\rangle}
\def\m{\mu}

\begin{document}

\begin{center}

{\huge{\bf Smash coproducts of bicomonads and Hom-entwining structures}}

\end{center}

\ \\
\begin{center}
{\bf Xiaohui ZHANG$^{1\ast}$, Wei WANG$^2$, Xiaofan ZHAO$^3$
}
\end{center}

\ \\
\hspace{-0,5cm}$^{1}$ Corresponding author, E-mail: zxhhhhh@hotmail.com.  School of Mathematical Sciences, Qufu Normal University, Qufu Shandong 273165, P. R. China.\\
\hspace{-0,5cm}$^{2}$ Department of Mathematics, Southeast University, Nanjing Jiangsu 210096, P. R. China.\\
\hspace{-0,5cm}$^{3}$ College of Mathematics and Information Science, Henan Normal University,
Xinxiang Henan 453007, P. R. China.\\

{\bf Abstract} Let $F,G$ be bicomonads on a monoidal category $\mathcal{C}$. The aim of this paper is to discuss the smash coproducts of $F$ and $G$. As an application,
the smash coproduct of Hom-bialgebras is discussed. Further, the Hom-entwining structure and Hom-entwined modules are investigated.
\vspace{0.5cm}



\section{\textbf{Introduction}}

The study of Hom-type algebras arises from the research on Witt and Virasoro type algebras.
In 2006, Hartwig, Larsson and Silvestrov introduced the Hom-Lie algebras which are related to $q$-deformations of Witt and Virasoro algebras (see \cite{JDS}). In a Hom-Lie algebra, the Jacobi identity is replaced by the so called Hom-Jacobi identity via a homomorphism.
In 2008, Makhlouf and Silvestrov (see \cite{AS2}) introduced the definition of Hom-associative algebras,
where the associativity of a Hom-algebra is twisted by an endomorphism (here we call it the Hom-structure map). The
generalized notions, including Hom-bialgebras, Hom-Hopf algebras were developed in \cite{af2}, \cite{AS1}, \cite{AS4}.
Further research on various Hom-Lie structures and Hom-type algebras can be found in \cite{Ag}, \cite{gmmp}, \cite{dy30}, and so on. Quasitriangular Hom-bialgebras were introduced by D. Yau (\cite{dy2}, \cite{dy3}), to provided a solution of the quantum Hom-Yang-Baxter equation, a twisted version of the quantum Yang-Baxter equation (\cite{dy4}, \cite{dy5}).
The Hom-Yetter-Drinfeld modules were investigated by Makhlouf and Panaite in\cite{af2}, \cite{AF} and \cite{CoDr}.

An interesting question is to explain Hom-type algebras use the theory of monoidal categories. In 2011, in order to provide a categorical approach to Hom-type algebras, Caenepeel and Goyvaerts (\cite{CG}) introduced the notions of Hom-categories and monoidal Hom-Hopf algebras. In a Hom-category, the associativity and unit constraints are twisted by the Hom-structure maps. A (co)monoid in a Hom-category is a Hom-(co)algebra, and a bimonoid in a Hom-category is a monoidal Hom-bialgebra.
Note that the main feature of a monoidal Hom-bialgebra is that the Hom-structure maps over the algebra structure and the coalgebra structure are invertible with each other, hence a monoidal Hom-bialgebra is a Hom-bialgebra if and only if the Hom-structure map $\a$ satisfies $\a^2=id$.

The motivation of this paper is to define and study the Hom-type entwining structures of a Hom-algebra and a Hom-coalgebra.
Entwining structure is proposed by T.Brzezinski and S.Majid in \cite{ts} in defining coalgebra principal bundles. The
relevance of entwining structures is raised by the observation of
M. Takeuchi that they provide examples of corings. An entwining structure over a monoidal category $\mathcal{C}$ consists of an algebra $A$, a coalgebra
$C$ and a morphism $\varphi:C\o A\rightarrow A \o C$ satisfying some axioms (\cite{MB}). The entwining modules are both $A$-modules and $C$-comodules, with compatibility relation given by $\varphi$. Note that the definition of entwined modules generalizes lots of important modules such as Hopf modules, Doi-Hopf modules, Long-dimodules and Yetter-Drinfeld modules.

Since a monoidal Hom-(co)algebra is a (co)monoid in the so-called Hom category (in \cite{CG}), it is not hard to realize that the entwining structure of a monoidal Hom-algebra and a monoidal Hom-coalgebra is actually an entwining structure in a non-strict monoidal category (see \cite{sk}, or \cite{gzw}). However, the Hom-type entwining structures over the Hom-type (not monoidal Hom-type) algebras and coalgebras are more complicated, because such a monoidal category $\mathcal{C}$ is not exist: the Hom-(co)algebra is a (co)monoid in $\mathcal{C}$.

Recall from \cite{sgz1} that an entwining structure $(C,A,\varphi)$ is in correspondence with the smash product structure $(A,C^\ast,R)$ (see \cite{sgz1}, Theorem 8), or equivalently, in correspondence with the smash coproduct structure $(A^\ast,C,V)$ (see \cite{sgz1}, Theorem 12). Note that in a monoidal category $\mathcal{C}$, if $X$ and $Y$
are (co)algebras, $S:X \o Y \rightarrow Y \o X$ is a morphism in $\mathcal{C}$, then
$\_ \o X$ and $\_ \o Y$ are all (co)monads on $\mathcal{C}$, and
$(X,Y,S)$ is a so-called smash (co)product structure if and only if $S$ induces a (co)monad distributive law
$\_ \o S: \_ \o X \o Y \Rightarrow \_ \o Y \o X$ in $\mathcal{C}$ (see \cite{CG1}, \cite{Rw}). Consequently, we can get that an entwining structure is in correspondence with the (co)monad distributive law.
Inspired by this conclusion, we can define the Hom-type entwining structure through a similar way.

The research on monads and comonads began in the 1950's. The earlier article about the theory of monads can be found in \cite{jb}, \cite{Se} and \cite{Sm} and so on.
In 2002, Moerdijk (\cite{Moe}) used a comonoidal monad to define a bimonad. Although Moerdijk called his bimonad "Hopf monad", the antipode was not involved in his definition. In 2007, A.Bruguieres and A.Virelizier (\cite{ABAV}) introduced the notion of Hopf monad with antipode in another direction, which is different from Moderijk. Because of their close connections with the monoidal structures, the theory of A.Bruguieres and A.Virelizier had developed rapidly and got many fundamental achievements (see \cite{ABA}, \cite{AB}, \cite{G2}, etc.).

In 2015, Zhang and Wang (see \cite{zw}) showed that the tensor functor $\_ \o H$ of a Hom-bialgebra $H$ is a bimonad on a special monoidal category, and a bicomonad on another monoidal category.
Hence we can use the theory of (co)monads to define the smash (co)products of Hom-(co)algebras, and the Hom-type entwining structures.
Therefore one is prompted to ask several questions:

$\bullet$ Does a (co)monad distributive law induced a smash (co)product of Hom-(co)algebras?

$\bullet$ Is it possible to characterize Hom-type entwining structures obtained from (co)monad distributive laws or the smash (co)products of Hom-(co)algebras?

$\bullet$ What are the Hom-entwining structures and Hom-entwined modules? Are the same as the monoidal Hom-case?

$\bullet$ When does the category of the Hom-entwined modules becomes a monoidal category?

$\bullet$ Does the Hom-entwined modules can be view as the generalization of some important Hom-type modules such as
Hom-Yetter-Drinfeld modules, Hom-Long dimodules?

The aim of this paper is to answer these questions.
In order to investigate these questions, we always assume that the Hom-(co)algebras in this paper is (co)unital, and the Hom structure map $\a$ and its (co)representations are all bijective maps.

The paper is organized as follows. In Section 2 we recall some basic notions such as comonads, comonad distributive laws and Hom-(co)algebras.
In section 3, we use the theory of 2-categories to give some necessary and sufficient conditions for the category of bicomodules of a comonad distributive law being monoidal.
In section 4, we mainly show that the comonad distributive laws can deduce the smash coproduct of two Hom-coalgebrs $H$ and $B$, and
describe the Hom-coalgebra structure of the smash coproduct $B \o H$. We also give the necessary and sufficient conditions for $B \o H$ being a Hom-bialgebra.
In section 5,
We take $B = {A^{\ast}}^{cop}$ to investigate the definition of the Hom-entwining structures and the $n$-th Hom-entwined modules, and discuss the monoidal structure on the category of $n$-th Hom-entwined modules. We find that $\mathcal{M}(\Phi)^H_A(n)$ is isomorphic to the corepresentations of the Hom-type smash coproduct structures. The notion of Hom-monoidal entwining structures is also introduced, and we prove that $\mathcal{M}(\Phi)^H_A(n)$ is a monoidal category if and only the Hom-entwining structure is actually a Hom-monoidal entwining structure.
Finally, in section 6, we consider the Hom-Doi-Hopf modules, Hom-Yetter-Drinfeld modules and Hom-Long dimodules for application. The Hom-type $\mathcal{D}$-Equation and Hom-type Yang-Baxter Equation are also discussed.

\section{\textbf{Preliminaries}}
\def\theequation{2.\arabic{equation}}
\setcounter{equation} {0} \hskip\parindent

Throughout the paper, we let $\mathbb{Z}$ be the set of all integers,
let $\k$ be a fixed
field and $char(\k)=0$ and $Vec_\k$ be the category of finite dimensional $\k$-spaces. All the (Hom-type) algebras and coalgebras, (Hom-type) modules and comodules are supposed to be in $Vec_\k$. For the comultiplication
 $\D $ of a $\k$-module $C$, we use the Sweedler-Heyneman's notation: $\Delta(c)=c_{1}\o c_{2}$ for any $c\in C$.

\vskip 0.5cm
 {\bf 2.1. Comonads and bicomonads.}
\vskip 0.5cm

Let $\mathcal {C}$ be a category, $F$: $\mathcal {C}$$\rightarrow$$\mathcal {C}$ be a functor. Recall from \cite{jb} or \cite{ROS} that if there exist natural transformations $\Delta$: $F \rightarrow FF$ and $\varepsilon$: $F \rightarrow id_{\mathcal {C}}$, such that the following identities hold
$$
F\delta \circ \Delta = \delta F\circ \Delta,~~\mbox{and}~~id_F = F\varepsilon\circ \Delta = \varepsilon F\circ \Delta,
$$
then we call the triple $(F, \Delta, \varepsilon)$ a \emph{comonad} on $\mathcal {C}$.

Let $X \in \mathcal {C}$, and $(F, \Delta, \varepsilon)$ a comonad on $\mathcal {C}$. If there exists a morphism $\rho^X$: $X$$\rightarrow$$FX$, satisfying
$$
G{\rho^X}\circ \rho^X = \Delta_X\circ \rho^X ,~~\mbox{and}~~\varepsilon_B\circ \rho^X = id_X,
$$
then we call the couple $(X, \rho^X)$ an \emph{F-comodule}.

A morphism between $F$-comodules $g$: $X \rightarrow X'$ is called $F$-colinear in $\mathcal {C}$, if $g$ satisfies: $Fg \circ \rho^X = \rho^{X'} \circ g$. The category of $F$-comodules is denoted by $\mathcal {C}^F$.

Let $\mathcal {C}$ be a category on which $(F, \Delta, \varepsilon)$ and $(G, \delta, \epsilon)$ are two comonads. A natural transformation $\varphi$: $FG$$\rightarrow$$GF$ is called a \emph{comonad distributive law}, if $\varphi$ induces the following commutative diagrams:
$$\aligned
\xymatrix
{
FG\ar[d]_{\varphi}\ar@{}"2,3"|{(L1)}\ar"1,2"^-{F \delta} & FGG \ar[r]^{\varphi G} & GFG \ar[d]^{G\varphi} \\
GF\ar"2,3"_{\delta F} & & GGF,
}~~
\xymatrix{
FG \ar@{}"2,3"|{(L2)}\ar[d]_{\varphi} \ar[r]^{\Delta G} & FFG \ar[r]^{F\varphi } & FGF \ar[d]^{\varphi F} \\
   GF \ar"2,3"_{G\Delta} &  & GFF, }
\endaligned$$
$$\aligned
\xymatrix{
FG \ar[r]^{\varphi} \ar[dr]_{\varepsilon G}^{(L3)} & GF \ar[d]^{G\varepsilon}  \\
 & G, }~~~~
\xymatrix{
FG \ar[r]^{\varphi} \ar[d]_{F\epsilon} & GF \ar[dl]^{\epsilon F}_{(L4)} \\
F }
\endaligned$$

\begin{remark}
Recall from \cite{GTR}, we know that $\varphi$ is a comonad distributive law if and only if one of the following statements hold:

(1). there is a lifting $\hat{F}$:$\mathcal {C}^G \rightarrow \mathcal {C}^G$ induced by $F$;

(2). $(FG, F\varphi G \circ \Delta\delta, \varepsilon\epsilon)$ is a comonad on $\mathcal {C}$.
\end{remark}

\begin{definition}
If $\varphi: FG\rightarrow GF$ is a comonad distributive law, then we call the comonad $FG=(FG, F\varphi G \circ \Delta\delta, \varepsilon\epsilon)$ the \emph{smash coproduct} of $F$ and $G$.
\end{definition}


Let $\mathcal {C}$ be a category, $\varphi:FG\rightarrow GF$ be a comonad distributive law, $M \in \mathcal {C}$, $(M, \theta^M)$ an $F$-comodule, and $(M, \rho^M)$ a $G$-comodule. If the diagram
$$\aligned
\xymatrix{
M \ar[d]_{\rho^M} \ar[r]^{\theta^M} & FM \ar[r]^{F(\rho^M)} & FGM \ar[d]^{\varphi_M}  \\
   GM \ar"2,3"^{G(\theta^M)} &  & GFM   }
\endaligned$$
is commutative, then we call the triple $(M, \theta^M, \rho^M)$ an \emph{$(F,G)$-bicomodule} or \emph{$\varphi$-bicomodules}.

A morphism between two $\varphi$-bicomodules is called a \emph{bicomodule morphism} if it is both $F$-colinear and $G$-colinear. The category of $\varphi$-bicomodules
 is denoted by $\mathcal {C}^{(F,G)}(\varphi)$.

\begin{remark}
Recall from \cite{GTR}, we know that the $(F,G)$-bicomodule is equivalent to $FG$-comodule.
\end{remark}

Let $(\mathcal {C}, \otimes, I,a,l,r)$ be a monoidal category, $(G, \delta, \epsilon)$ a comonad on $\mathcal {C}$, and $G$ also a monoidal functor, i.e., there exists a natural transformation $G_2$: $G\otimes G \rightarrow G\otimes$ and a morphism $G_0$: $I \rightarrow G(I)$ in $\mathcal {C}$, such that for any $X, Y, Z \in \mathcal {C}$, the following equations hold:
\begin{eqnarray*}
&&G_2(X,Y\otimes Z) \circ (id_{G(X)} \otimes G_2(Y, Z))\circ a_{GX,GY,GZ}\\
&&~~~~~~~~~~ = G(a_{X,Y,Z})\circ G_2(X\otimes Y, Z)\circ(G_2(X,Y)\otimes id_{G(Z)} ),\\
 &&G(r_X)\circ G_2(X,I)\circ(id_{G(X)}\otimes G_0)\circ r^{-1}_{GX} \\
 &&~~~~~~~~~~= id_{G(X)} = G(l_X)\circ G_2(I,X)\circ(G_0 \otimes id_{G(X)})\circ l^{-1}_{GX} .
\end{eqnarray*}
Then recall from \cite{ABAV} that $G$ is called a \emph{bicomonad} (or a \emph{monoidal comonad}) on $\mathcal {C}$ if the following identities hold
\\
$\left\{\begin{array}{l}
(C1)~~G(G_2(X,Y)) \ci G_2(GX,GY) \ci (\delta_X \o \delta_Y) = \delta_{X \o Y} \ci G_2(X,Y);\\
(C2)~~\epsilon_{X \o Y} \ci G_2(X,Y) = \epsilon_X \o \epsilon_Y;\\
(C3)~~ G(G_0) \ci G_0 = \delta_I \ci G_0;\\
(C4)~~\epsilon_I\circ G_0 = id_I.
\end{array}\right.$
\\

From the above definition one can see that $(I, G_0) \in \mathcal {C}^G$.

\begin{remark}
Suppose that $(G,\delta,\epsilon)$ is a comonad, and $(G,G_2,G_0)$ is a monoidal functor on a monoidal category $(\mathcal {C},\o,I,a,l,r)$. If we define the $G$-coaction on $I$ by $G_0$, and define the following monoidal structure
$$\rho^{M \o N}:
\xymatrix{
M \o N \ar[rr]^{\rho^M \o \rho^N} && GM \o GN \ar[rr]^{G_2(M,N)} && G(M \o N),
}
$$
for any $(M, \rho^M), (N,\rho^N) \in \mathcal {C}^G$,
then $(\mathcal {C}^G, \o ,I,a,l,r)$ is a monoidal category
if and only if $G=(G,\delta,\varepsilon,G_2,G_0)$ is a bicomonad on $\mathcal {C}$.
\end{remark}

\vskip 0.5cm
 {\bf 2.2. Hom-bialgebras and Hom-Hopf algebras}
\vskip 0.5cm

In this section, we will review several definitions and notations related to (finite dimensional) Hom-bialgebras.

Recall from \cite{jjr} that
a \emph{Hom-algebra} over $\k$ is a quadruple $(A,\mu,1_A,\alpha)$, in which $A$ is a $k$-module, $\alpha:A\rightarrow A$, $\mu: A \o A\rightarrow A$ are linear maps, and $1_A\in A$, satisfying the following conditions, for all $a,b,c \in A$:
\begin{eqnarray*}
&\a(a)(bc) = (ab)\a(c),~~~~\a(1_A) = 1_A,~~~~1_Aa = a1_A = \a(a).
\end{eqnarray*}

Recall from \cite{jjr} that
a \emph{Hom-coalgebra} over $\k$ is a quadruple $(C,\a,\Delta,\epsilon)$, in which $C$ is a $k$-module, $\a:C\rightarrow C$, $\Delta: C\rightarrow C\o C$ and $\epsilon:C\rightarrow \k$ are linear maps, with notation $\Delta(c) = c_1 \o c_2$, satisfying the following conditions for all $c \in C$:
\begin{eqnarray*}
&\epsilon \circ \a = \epsilon,~~~~\a(c_1) \o \Delta(c_{2}) = \Delta(c_1) \o \a(c_2),~~~~\epsilon(c_1)c_2 = c_1\epsilon(c_2) = \a(c).
\end{eqnarray*}

Note that in the earlier definition of Hom-(co)algebras by Makhlouf and Silvestrov (see \cite{AS1} or \cite{AS2}),
an axiom was redundant as shown in \cite{jjr}. The reader will easily
check that the definition above is equivalent to the one in those papers.

 Let $(H,\alpha)$ be a Hom-algebra.  A left \emph{$(H,\alpha)$-Hom-module} is a triple $(M,\a_M, \theta_M)$, where $M$ is a $\k$-space, $\a_M:M\rightarrow M$ and $\theta_M:H \o M\rightarrow M$ are linear maps with notation $\theta_M(h \o m) = h\cdot m$, satisfying the following conditions, for all $h,h'\in H$, $m\in M$,
\\
$\left\{\begin{array}{l}
(1)~~\a(h)\cdot(h'\cdot m) = (hh')\cdot\a_M(m);\\
(2)~~1_H\cdot m = \a_M(m).
\end{array}\right.$
\\

A \emph{morphism $f:M\rightarrow N$ of Hom-modules} is a $\k$-linear map such that $\theta_N\circ (id_H \o f) = f\circ \theta_M$.

Let $C$ be a Hom-coalgebra. Recall that a \emph{right $C$-comodule} is a triple $(M,\a_M, \rho^M)$, where $M$ is a $\k$-module, $\a_M:M\rightarrow M$ and $\rho^M:M\rightarrow M \o C$ are linear maps with notation $\rho^M(m) = m_0\o m_1$, satisfying the following conditions for all $m \in M$:\\
$\left\{\begin{array}{l}
(1)~~\a_M(m_0)\o \Delta(m_1) = \rho^M(m_0) \o \a(m_1);\\
(2)~~\epsilon(m_1)m_0 = \a_M(m).
\end{array}\right.$
\\

\emph{A morphism $f:M\rightarrow N$ of $C$-comodules} is a linear map such that $\rho^N\circ f=(id_C \o f)\circ \rho^M$.

Recall that in the earlier definition of Hom-(co)modules by Makhlouf and Silvestrov, there is also a redundant axiom (see \cite{jjr} for details).

Recall from \cite{AS4} that
a \emph{Hom-bialgebra} over $\Bbbk$ is a sextuple $H=(H,\mu,\eta,\Delta,\varepsilon,\a)$, in which $(H, \a,
\m,\eta)$ is a Hom-algebra, $(H,\alpha,\Delta,\varepsilon)$ is a Hom-coalgebra,
and $\D,\v$ are morphisms of Hom-algebras preserving unit.

Recall from \cite{AS4} that
a \emph{Hom-Hopf algebra} over $\Bbbk$ is
a Hom-bialgebra $(H, \alpha)$ together with a $\Bbbk$-linear map
$S:H\rightarrow H$ (called \emph{the antipode}) such
that
$$S\ast id=id\ast S=\eta\varepsilon,~S\circ\alpha=\alpha\circ S.$$

Let $(H, \a)$ be a Hom-Hopf algebra. Note that if $\a$ is invertible, then for all $a, b \in H$, $S$ satisfies
\begin{eqnarray*}
&S(ab) = S(b)S(a),~~S(1_H) = 1_H,\\
&\Delta(S(a)) = S(a_2) \o S(a_1),~~\varepsilon\circ S=\varepsilon.
\end{eqnarray*}

\section{\textbf{The monoidal structure of $\mathcal {C}^{(F,G)}(\varphi)$}}
\def\theequation{3.\arabic{equation}}
\setcounter{equation} {0} \hskip\parindent

Throughout this section, assume that
$(\mathcal {C}, \otimes, I, a, l, r)$ is a monoidal category on which $(F, \Delta, \varepsilon)$ and $(G, \delta, \epsilon)$ are bicomonads, and $\varphi:FG\rightarrow GF$ is a comonad distributive law.

Notice that for any $X \in \mathcal {C}$, if we define
$$\theta^{FGX}:
\xymatrix@C=0.5cm{
   FGM \ar[rr]^{\Delta_{GM}} && FFGM } $$
and
$$\rho^{FGX} :
\xymatrix@C=0.5cm{
  FGX \ar[rr]^{F(\delta_X)} && FGGX \ar[rr]^{\varphi_{GX}} && GFGX}, $$
then it is easy to check that $(FGX, \theta^{FGX}, \rho^{FGX}) \in \mathcal {C}^{(F,G)}(\varphi)$.

\begin{lemma}
Let $(X, \theta^X, \rho^X)$ and $(Y, \theta^Y, \rho^Y)$ be objects in $\mathcal {C}^{(F,G)}(\varphi)$. If the $F$-coaction $\theta^{X\otimes Y}$ and $G$-coaction $\rho^{X\otimes Y}$ on $X\otimes Y$ are given by
$$\theta^{X\otimes Y} :
\xymatrix@C=0.5cm{
X\otimes Y\ar[rr]^{\theta^X\otimes \theta^Y} && FX \otimes FY \ar[rr]^{F_2(X,Y)} && F(X\otimes Y)} $$
and
$$\rho^{X\otimes Y} :
\xymatrix@C=0.5cm{
   X\otimes Y\ar[rr]^{\rho^X\otimes \rho^Y} && GX\otimes GY \ar[rr]^{G_2(X,Y)} && G(X \otimes Y)}, $$
then $(\mathcal {C}^{(F,G)}(\varphi),\otimes,I,a,l,r)$ is a monoidal category if and only if $(F, G, \varphi)$ satisfies the following equations for any $M,N \in \mathcal {C}$:

a) $\varphi_{M\otimes N}\circ F(G_2(M,N))\circ F_2(GM,GN) = G(F_2(M,N))\circ G_2(FM,FN)\circ (\varphi_M \otimes \varphi_N)$;

b) $\varphi_I \circ F(G_0) \circ F_0 = G(F_0)\circ G_0$.
\end{lemma}

\begin{proof}
$\Rightarrow)$: By the assumption, we have $(FGM\otimes FGN,\theta^{FGM\otimes FGN}, \rho^{FGM\otimes FGN}) \in \mathcal {C}_F^G(\varphi)$ for any $M,N \in \mathcal {C}$, i.e.
\begin{eqnarray*}
&&~~~~G(F_2(FGM,FGN))\circ G(\Delta_{GM}\otimes\Delta_{GN})\circ G_2(FGM,FGN)\circ (\varphi_{GM}\otimes\varphi_{GN})\circ (F\delta_M \otimes F\delta_N) \\
&&= \varphi_{FGM\otimes FGN}\circ F(G_2(FGM,FGN))\circ F(\varphi_{GM}\otimes \varphi_{GN})\circ F(F\delta_M \otimes F\delta_N)\circ F_2(FGM,FGN)  \\
&&~~~~\circ (\Delta_{GM}\otimes\Delta_{GN}).
\end{eqnarray*}
Multiplied by $GF(\epsilon_M\otimes\epsilon_N)\circ GF(\varepsilon_{GM}\otimes\varepsilon_{GN})$ right on both sides of the above identity, we immediately get the conclusion a). Since $(I,F_0,G_0) \in \mathcal {C}^{(F,G)}(\varphi)$, one can see that b) holds.

$\Leftarrow)$: For any $M,N \in \mathcal {C}^{(F,G)}(\varphi)$, it is easy to show that $(M\otimes N, \theta^{M\otimes N}) \in \mathcal {C}^F$ and $(M\otimes N, \rho^{M\otimes N}) \in \mathcal {C}^G$.
Then from the following commutative diagram
$$\aligned
\xymatrix{
M\otimes N \ar[rr]^{\theta^M\otimes \theta^N}\ar[d]_{\rho^M\otimes \rho^N} & & FM\otimes FN \ar[r]^{F_2(M,N)} \ar[d]|-{F\rho^M\otimes F\rho^N}  & F(M \otimes N)  \ar[d]^{F(\rho^M\otimes \rho^N)}\\
GM\otimes GN \ar[d]_{G_2(M,N)} \ar[r]^{G\theta^M\otimes G\theta^N} &  GFM\otimes GFN \ar[d]_{G_2(FM,FN)} & FGM\otimes FGN \ar[l]^{\varphi_M \otimes \varphi_N}  \ar[r]_{F_2(GM,GN)} & F(GM \otimes GN) \ar[d]^{F(G_2(M,N))}\\
G(M\otimes N) \ar[r]_{G(\theta^M\otimes \theta^N)} & G(FM\otimes FN) \ar[r]_{G(F_2(M,N))} & GF(M\otimes N) & FG(M\otimes N)\ar[l]^{\varphi_{M \otimes N}},}
 \endaligned$$
we get $(M\otimes N,\theta^{M\otimes N}, \rho^{M\otimes N}) \in \mathcal {C}_F^G(\varphi)$.

From the assumption b), one can easily get $I\in \mathcal {C}^{(F,G)}(\varphi)$.
\end{proof}

\begin{definition}
We call $\varphi$ is a \emph{monoidal comonad distributive law} if the comonad distributive law $\varphi:FG\rightarrow GF$ satisfies condition a) and b) in Lemma 3.1.
\end{definition}

Recall from \cite{JH} and \cite{ROS}, if $\mathbb{C}$ denotes any 2-category, then the following data forms the 2-category of comonads, which is denoted by $\mathbf{Cmd(\mathbb{C})}$ :

$\bullet$ the 0-cell contains an object $Y$, a 1-cell $T:Y\rightarrow Y$ in $\mathbb{C}$, together with the comultiplication $\delta: T\rightarrow TT$, and the counit $\epsilon:T\rightarrow 1_Y$, which satisfies the coassociative law and the counit law, respectively;

$\bullet$ the 1-cell in $\mathbf{Cmd(\mathbb{C})}$ from $(Y,T,\delta,\epsilon)$ to $(Y',T',\delta',\epsilon')$ is a 1-cell $W: Y\rightarrow Y'$ in $\mathbb{C}$ together with a 2-cell $w: WT\Rightarrow T'W$ in $\mathbb{C}$, satisfying
\begin{eqnarray*}
\delta'W \ci w = T'w \ci wT \ci W\delta,~~~~\mbox{and~~~~}\epsilon'W \ci w = W\epsilon;
\end{eqnarray*}

$\bullet$ the 2-cell in $\mathbf{Cmd(\mathbb{C})}$ from $(W,w)$ to $(V,v)$ is a 2-cell $\chi: W\Rightarrow V$ in $\mathbb{C}$ which satisfies
\begin{eqnarray*}
v \ci \chi T = T'\chi \ci w.
\end{eqnarray*}

Similarly, the following data forms a 2-category $\mathbf{CC(\mathbb{C})}$ of the distributive laws:

$\bullet$ the 0-cell $(C,D,T,\varphi)$ consists of an object $C$ of $\mathbb{C}$, $(D,\Delta,\varepsilon)$ and $(T,\delta,\epsilon)$ are comonads on $C$, $\varphi:DT\Rightarrow TD$ is a comonad distributive law;

$\bullet$ the 1-cell $(J,j_d,j_t):(C,D,T,\varphi)\rightarrow(C',D',T',\varphi')$ consists of a 1-cell $J: C\rightarrow C'$ in $\mathbb{C}$, together with 2-cells $j_d: JD \Rightarrow D'J$ and $j_t:JT\Rightarrow T'J$ where $j_d$ and $j_t$ satisfies
\begin{eqnarray*}
&&\Delta' J\circ j_d = D'j_d\circ j_dD\circ J\Delta ,~~~~  J\varepsilon = \varepsilon'J \circ j_d, \\
&&\delta'J \circ j_t = T'j_t\circ j_tT\circ J\delta ,~~~~  J\epsilon = \epsilon'J \circ j_t,
\end{eqnarray*}
and the following diagram:
$$\aligned
\xymatrix{
JDT \ar[d]_{j_dT}\ar[r]^{\varphi J} & JTD \ar[r]^{j_t D} & T'J D \ar[d]^{T'j_d}\\
D'JT \ar[r]^{D' j_t} & D'T'J \ar[r]^{{\varphi}' J} & T'D'J . }
\endaligned$$

$\bullet$ the 2-cell $\beta:(J,j_d,j_g)\Rightarrow(H,h_d,h_g)$, where $\beta: J\Rightarrow H$ is a 2-cell in $\mathbb{C}$, and satisfies:
$$\aligned
\xymatrix{
JD \ar[r]^{j_d} \ar[d]_{\beta D} & D'J \ar[d]^{D' \beta} \\
HD \ar[r]^{h_d} & D'H , }~~~~~~
\xymatrix{
JT \ar[r]^{j_t} \ar[d]_{\beta T} & T'J \ar[d]^{T' \beta} \\
HT \ar[r]^{h_t} & T'H. }
\endaligned$$

Let $\mathbb{C} = \mathbf{Cat}$, then we get the following property.

\begin{proposition}
The following statements are equivalent.

(1) $\varphi:FG\rightarrow GF$ is a monoidal comonad distributive law;

(2) $F_2:(\otimes (F\times F),G_2(F\times F)\circ \otimes(\varphi\times\varphi))\Rightarrow (F\otimes, \varphi\otimes \circ FG_2): (\mathcal{C}\times \mathcal{C}, G\times G, \delta\times\delta, \epsilon\times\epsilon) \rightarrow (\mathcal{C}, G, \delta, \epsilon)$ and $F_0:(I,G_0)\Rightarrow (FI,\varphi_I\circ FG_0): (\mathfrak{I},id,id,id)\rightarrow(\mathcal{C}, G, \delta, \epsilon)$ are 2-cells in the 2-category $\mathbf{Cmd(\mathbb{C})}$, where $\mathfrak{I}$ means the terminal category;

(3) $(\otimes,F_2,G_2):(\mathcal{C}\times \mathcal{C},F\times F,G\times G,\varphi\times \varphi)\rightarrow (\mathcal{C},F,G,\varphi)$ and $(I,F_0,G_0): (\mathfrak{I},id,id,id)\rightarrow (\mathcal{C},F,G,\varphi)$ are 1-cells in $\mathbf{CC(\mathbb{C})}$.
\end{proposition}

\begin{proof}
Straightforward.
\end{proof}

\begin{lemma}
$\varphi$ is a monoidal comonad distributive law if and only if the smash coproduct $FG$ is a bicomonad on $\mathcal{C}$.
\end{lemma}

\begin{proof}
Recall from Remark 2.1 that $FG$ is a comonad on $\mathcal{C}$. Since $F$ and $G$ are both monoidal functors, $FG$ is also a monoidal functor: for any
$M, N\in \mathcal {C}$,
$$(FG)_2(M,N) = F(G_2(M,N))\circ F_2(GM,GN);~~~~(FG)_0 = F(G_0)\circ F_0.$$

$\Rightarrow$: If $FG$ is a bicomonad, thus for $M, N\in \mathcal {C}$, we have
\begin{eqnarray*}
&&FGF(G_2(M,N))\circ FG(F_2(GM,GN))\circ F(G_2(FGM,FGN))\circ F_2(GFGM,GFGN) \\
&&~~~~\circ (F\varphi_{GM}\otimes F\varphi_{GN}) \circ (\Delta\delta_M\otimes \Delta\delta_N) \\
&&= F\varphi_{G(M\otimes N)}\circ \Delta\delta_{M\otimes N}\circ F(G_2(M,N))\circ F_2(GM,GN).
\end{eqnarray*}
Multiplied by $\varepsilon GF\epsilon_{(M\otimes N)}$ right on both sides of the above identity, we immediately get the conclusion a). Similarly, one can see that b) holds.

$\Leftarrow$: We only check Equations (C1) and (C3). Firstly, for any $M, N\in \mathcal {C}$, we have
\begin{eqnarray*}
&&~~ FGF(G_2(M,N)) \ci FG(F_2(GM,GN)) \ci F(G_2(FGM,FGN)) \ci F_2(GFGM,GFGN) \\
&&~~~~~~\ci (F\varphi_{GM}\otimes F\varphi_{GN}) \ci (FF\delta_M\otimes FF\delta_N) \ci (\Delta_{GM}\otimes \Delta_{GN}) \\
&&=FGF(G_2(M,N)) \ci F\varphi_{GM\otimes GN} \ci FF(G_2(GM,GN)) \ci F(F_2(GGM,GGN)) \\
&&~~~~\ci F_2(FGGM,FGGN) \ci (FF\delta_M\otimes FF\delta_N) \ci (\Delta_{GM}\otimes \Delta_{GN}) \\
&&=FGF(G_2(M,N)) \ci F\varphi_{GM\otimes GN} \ci FF(G_2(GM,GN)) \ci FF(\delta_M\otimes \delta_N) \ci \Delta_{GM\otimes GN} \ci F_2(GM,GN)\\
&&=F\varphi_{G(M\otimes N)} \ci (FF\delta_{M\otimes N}) \ci \Delta_{G(M\otimes N)} \ci F(G_2(M,N)) \ci F_2(GM,GN),
\end{eqnarray*}
thus (C1) holds.

Secondly, consider the following commutative diagram
$$\aligned
\xymatrix{
I \ar[r]^{F_0} \ar[d]_{F_0} & FI \ar[d]_{F(F_0)}\ar[r]^{F(G_0)} & FGI \ar[r]^{FG(F_0)} & FGFI \ar[dd]|{FGF(G_0)}\\
FI \ar[r]_{\Delta_I}\ar[d]_{F(G_0)} & FFI\ar[d]|{FF(G_0)}\ar[r]^{FF(G_0)} & FFGI \ar[d]|{FFG(G_0)}\ar"1,4"|{F\varphi_I} & \\
FGI \ar[r]_{\Delta_{GI}} & FFGI \ar[r]_{FF\delta_I} & FFGGI \ar[r]_{F\varphi_{GI}} & FGFGI ,  }
\endaligned$$
(C3) holds. (C2) and (C4) can be proved similarly.
\end{proof}

Combining Lemma 3.1, Proposition 3.2 and Lemma 3.3, we immediately get the following theorem.

\begin{theorem} Assume that
$(\mathcal {C}, \otimes, I, a, l, r)$ is a monoidal category, $(F, \Delta, \varepsilon)$ and $(G, \delta, \epsilon)$ are bicomonads, and $\varphi:FG\rightarrow GF$ is a comonad distributive law on $\mathcal {C}$. Then the following statements are equivalent:

(1) $(\mathcal {C}^{(F,G)}(\varphi),\otimes,I,a,l,r)$ is a monoidal category, where the monoidal structure is given in Lemma 3.1;

(2) $\varphi:FG\rightarrow GF$ is a monoidal comonad distributive law;

(3) the smash coproduct $FG$ is a bicomonad;

(4) $F_2:(\otimes (F\times F),G_2(F\times F)\circ \otimes(\varphi\times\varphi))\Rightarrow (F\otimes, \varphi\otimes \circ FG_2)$ and $F_0:(I,G_0)\Rightarrow (FI,\varphi_I\circ FG_0)$ are 2-cells in the 2-category $\mathbf{Cmd(\mathbb{C})}$, where $\mathfrak{I}$ means the terminal category;

(5) $(\otimes,F_2,G_2):(\mathcal{C}\times \mathcal{C},F\times F,G\times G,\varphi\times \varphi)\rightarrow (\mathcal{C},F,G,\varphi)$ and $(I,F_0,G_0): (\mathfrak{I},id,id,id)\rightarrow (\mathcal{C},F,G,\varphi)$ are 1-cells in $\mathbf{CC(\mathbb{C})}$.
\end{theorem}

\section{\textbf{The smash coproduct of Hom-bialgebras}}
\def\theequation{4.\arabic{equation}}
\setcounter{equation} {0} \hskip\parindent

At the beginning of this section, we introduce the following monoidal category $\overline{\mathcal{H}}^{i,j}(Vec_\Bbbk)$ for any $i,j \in \mathbb{Z}$:

$\bullet $ the objects of $\overline{\mathcal{H}}^{i,j}(Vec_\Bbbk)$ are pairs $(X,\a_X)$, where $X\in Vec_\Bbbk$ and $\a_X\in Aut_{\Bbbk}(X)$;

$\bullet $ the morphism $f:(X,\a_X)\rightarrow (Y,\a_Y)$ in $\overline{\mathcal{H}}^{i,j}(Vec_\Bbbk)$ is a $\Bbbk$-linear map from $X$ to $Y$ such that $\a_Y\circ f=f\circ \a_X$;

$\bullet $ the monoidal structure is given by
$$(X,\a_X)\otimes(Y,\a_Y)=(X\otimes Y,\a_X\otimes\a_Y),$$
and the unit is $(\Bbbk,id_{\Bbbk})$;

$\bullet $ the associativity constraint $a$ is given by
$$
a_{X,Y,Z}:(X \o Y) \o Z \rightarrow X \o(Y \o Z),~~(x \o y) \o z\mapsto \a^{i+1}_X(x) \o (y \o \a^{-j-1}_Z(z));
$$

$\bullet $ for any $x \in X \in Vec_\Bbbk$ and $\lambda \in \Bbbk$, the unit constraints $l$ and $r$ are given by
$$
l_X(\lambda \o x) = \lambda\a^{j+1}_X(x),~~~~l_X^{-1}(x) = 1_\Bbbk \o \a^{-j-1}_X(x),
$$
$$
r_X(x \o\lambda) = \lambda\a^{i+1}_X(x),~~~~r_X^{-1}(x) =\a^{-i-1}_X(x) \o 1_\Bbbk.
$$


Now, we assume that $(H, \a_H)$ is an object in $\overline{\mathcal{H}}^{i,j}(Vec_\Bbbk)$, $m_H: H\o H \rightarrow H$ (with notation $m_H(a \o b) = ab$),
$\eta_H: \k\rightarrow H$ (with notation $\eta_H(1_\k) = 1_H$), and $\Delta_H: H \rightarrow H\o H$ (with notation $\Delta_H(h) = h_1 \o h_2$),
and $\varepsilon_H: H\rightarrow \k$ are all morphisms in $\overline{\mathcal{H}}^{i,j}(Vec_\Bbbk)$. Further, we write
$$
\ddot{H}=\_\o H:\overline{\mathcal{H}}^{i,j}(Vec_\Bbbk)\rightarrow \overline{\mathcal{H}}^{i,j}(Vec_\Bbbk),~~~~(X, \a_X)\mapsto ( X \o H, \a_X \o \a_H)
$$
for the right tensor functor of $H$.

\begin{lemma}\label{cmd}
If we define the following structures on $\ddot{H}$:

$\bullet$ $\delta:\ddot{H}\rightarrow \ddot{H}\ddot{H}$ is defined by
$$
\delta_X: X \o H\rightarrow (X \o H) \o H,~~~x \o h \mapsto (\a_X(x) \o h_1) \o \a_H^{-1}(h_2);
$$

$\bullet$ $\epsilon:\ddot{H}\rightarrow id_{\overline{\mathcal{H}}^{i,j}(Vec_\Bbbk)}$ is given by
$$
\epsilon_X: X \o H\rightarrow  X,~~~ x \o h \mapsto \varepsilon_H(h) \a^{-1}_X(x),
$$
then $\ddot{H}=(\ddot{H},\delta, \epsilon)$ forms a comonad on
$\overline{\mathcal{H}}^{i,j}(Vec_\Bbbk)$ if and only if $(H,\a_H,\Delta_H,\varepsilon_H)$ is a Hom-coalgebra over $\k$. Further, $Corep(H)$ is exactly $\overline{\mathcal{H}}^{i,j}(Vec_\Bbbk)^{\ddot{H}}$.
\end{lemma}

\begin{proof}
$\Rightarrow$: For any $h \in H$, since the following diagram is commutative:
$$\aligned
\xymatrix{
x\o h \ar[rr]^-{\delta_X} \ar[d]_{\delta_X} && (\a_X(x) \o h_1) \o \a^{-1}_H(h_2) \ar[d]^{\delta_{X \o H}} \\
(\a_X(x) \o h_1) \o \a^{-1}_H(h_2) \ar[rr]^-{\delta_X \o id_H} && ((\a^2_X(x) \o \a_H(h_1)) \o \a_H^{-1}(h_{21})) \o \a_H^{-2}(h_{22}), }
\endaligned$$
take $x=1_\k$ and use $((\a_H^{-j-1} \o \a_H) \o \a^{-2}_H) \ci ((l_H\o id_H) \o id_H)$ to action at the both side of the identity, then we immediately get that $\a_H(h_1) \o h_{21} \o h_{22} = h_{11} \o h_{12} \o \a_H(h_2)$.

Similarly, one can show that $\v_H(h_1)h_2 = h_1 \v_H(h_2) = \a_H(h)$ through the counit law of $\ddot{H}$. Thus $H$ is a Hom-coalgebra.

$\Leftarrow$: See [\cite{zw}, Theorem 4.3].
\end{proof}

\begin{lemma}\label{mdl}
If we define the following structures on $\ddot{H}$:

$\bullet$ $\ddot{H}_2:\ddot{H} \o \ddot{H} \rightarrow \ddot{H} \o$ is given by
$$
\ddot{H}_2(X,Y):( X \o H) \o (Y \o H)\rightarrow (X \o Y) \o H,~~~(x \o a)\o (y \o b)\mapsto (x \o y) \o \a_H^i(a)\a_H^j(b),
$$
for any $ X, Y \in \overline{\mathcal{H}}^{i,j}(Vec_\Bbbk)$;

$\bullet$ $\ddot{H}_0: \Bbbk\rightarrow \ddot{H}(\Bbbk)$ is given by
$$
\ddot{H}_0:\Bbbk\rightarrow \Bbbk\o H ,~~~ \lambda\mapsto \lambda \o 1_H,
$$
then $\ddot{H}=(\ddot{H}_2, \ddot{H}_0)$ is a monoidal functor if and only if $(H,\a_H,m_H,\eta_H)$ is a Hom-algebra over $\k$.
\end{lemma}

\begin{proof}
$\Rightarrow$: For any $a,b,c \in H$, since $\ddot{H}$ is a monoidal functor, consider the following identity
\begin{eqnarray*}
&&(\ddot{H}_2(\k,\k\otimes \k) \circ (id_{\k \o H} \otimes \ddot{H}_2(Y, Z))\circ a_{ \k \o H, \k \o H,\k \o H})( ( (1_\k \o \a_H^{-2i}(a)) \o (1_\k \o \a_H^{-i-j}(b) ) )\\
&&~~~~\o ( 1_k \o \a_H^{-j+1}(c) ) )\\
&& = ( (a_{\k,\k,\k} \o H)\circ \ddot{H}_2(\k\otimes \k, \k)\circ(\ddot{H}_2(\k,\k)\otimes id_{\k \o H} )  ( ( (1_\k \o \a_H^{-2i}(a)) \o (1_\k \o \a_H^{-i-j}(b) ) ) \\
&&~~~~\o ( 1_k \o \a_H^{-j+1}(c) ) ),
\end{eqnarray*}
thus we have $\a_H(a)(bc) = (ab) \a_H(c)$.

We can get $a1_H = 1_H a = \a_H (a)$ in a similar way.

$\Leftarrow$: See [\cite{zw}, Theorem 4.3].
\end{proof}

\begin{theorem}
If we define $\delta$, $\epsilon$, and $\ddot{H}_2$, $\ddot{H}_0$ as in Lemma \ref{cmd} and Lemma \ref{mdl}, then
the following statements are equivalent:

(1) $\ddot{H}$ is a bicomonad on $\overline{\mathcal{H}}^{i,j}(Vec_\Bbbk)$;

(2) $(H,\a_H,m_H,\eta_H, \Delta_H,\varepsilon_H)$ is a Hom-bialgebra over $\k$;

(3) $Corep(H)$ is a monoidal category. Precisely, the monoidal structure in $Corep(H)$
is given by
\begin{eqnarray*}
\theta^{U \o V}(u \o v):= u_{(0)} \o v_{(0)} \o \a_H^i(u_{(1)})\a_H^j(v_{(1)}),
\end{eqnarray*}
where $U,V$ are all $H$-Hom-comodules, $u \in U$, $v \in V$. The associativity constraint and the unit constraint in $Corep(H)$ are same to $\overline{\mathcal{H}}^{i,j}(Vec_\Bbbk)$. We denote this monoidal category by $Corep^{i,j}(H)$. Further, $\overline{\mathcal{H}}^{i,j}(Vec_\Bbbk)^{\ddot{H}}$ is also a monoidal category, and $Corep^{i,j}(H)$ is identified to $\overline{\mathcal{H}}^{i,j}(Vec_\Bbbk)^{\ddot{H}}$ as a monoidal category.
\end{theorem}

\begin{proof}
(1)$\Rightarrow$(2):
We already know that $\ddot{H}$ is both a Hom-algebra and a Hom-coalgebra because Lemma \ref{cmd} and Lemma \ref{mdl}.

For any $a,b \in H$, since equation (C1) holds, we have
\begin{eqnarray*}
&& ( (\ddot{H}_2(\k,\k) \o id_H) \ci \ddot{H}_2(\k \o H, \k\o H) \ci (\delta_X \o \delta_Y)) ( ( 1_\k \o \a_H^{-i}(a) ) \o ( 1_\k \o \a_H^{-j}(b) ) ) \\
&& = (\delta_{\k \o \k} \ci \ddot{H}_2(\k,\k)) ( ( 1_\k \o \a_H^{-i}(a) ) \o ( 1_\k \o \a_H^{-j}(b) ) ),
\end{eqnarray*}
which implies $(ab)_1 \o (ab)_2 = a_1b_1 \o a_2 b_2$. Similarly, (C2) implies $\v_H$ preserve the multiplication, (C3) and (C4) mean $\Delta_H$ and $\v_H$ preserve the unit, respectively. Thus $H$ is a Hom-bialgebra.

(2)$\Rightarrow$(3): See [\cite{zw}, Theorem 6.2].

(3)$\Rightarrow$(1): See Remark 2.4.
\end{proof}

Now we suppose that $B=(B,\a_B)$ and $H=(H,\a_H)$ are two Hom-coalgebras over $\k$, $\ddot{B} = \_ \o B$, $\ddot{H} = \_ \o H$ mean the tensor functor defined as above.
Let $\varphi:B \o H\rightarrow H \o B$ (with notation $\varphi(b \o h) = \sum h^\varphi \o b^\varphi$) be a $\k$-linear map
satisfying
$$
\varphi \ci (\a_B \o \a_H) = (\a_H \o \a_B) \ci \varphi.
$$
For any $(X, \a_X) \in \overline{\mathcal{H}}^{i,j}(Vec_\Bbbk)$, if we define the natural transformation $\ddot{\varphi}(n)_X: \ddot{H}\ddot{B}(X)\rightarrow \ddot{B}\ddot{H}(X)$ by
\begin{eqnarray*}
\ddot{\varphi}(n)_X: (X \o B) \o H\rightarrow (X \o H) \o B,~~~~(x \o b) \o h\mapsto \sum (x \o {\a_H(h)}^\varphi) \o \a_B^{n}({\a_B^{-n-1}(b)}^\varphi),
\end{eqnarray*}
where $n$ is an integer, then we have the following property.

\begin{proposition}
$\ddot{\varphi}(n): \ddot{H}\ddot{B}\rightarrow \ddot{B}\ddot{H}$ is a comonad distributive law for any $n \in \mathbb{Z}$ if and only if the following equalities hold:
\\
$\left\{\begin{array}{l}
(M1)~ \sum  {{\a_H(h)}^\varphi}^\phi \o {b_1}^\phi \o {b_2}^\varphi = \sum \a_H(h^\varphi) \o {b^\varphi}_1 \o {b^\varphi}_2;\\
(M2)~ \sum {h_1}^\varphi \o {h_2}^\phi \o { {\a_B(b)}^\varphi}^\phi = \sum {h^\varphi}_1 \o {h^\varphi}_2 \o \a_B(b^\varphi);\\
(M3)~ \sum \v_H(h^\varphi) b^\varphi = \v_H(h) b;\\
(M4)~ \sum \v_B(b^\varphi) h^\varphi = \v_B(b) h.
\end{array}\right.$
\end{proposition}

\begin{proof}
$\Rightarrow$: Assume that $\ddot{\varphi}(n): \ddot{H}\ddot{B}\rightarrow \ddot{B}\ddot{H}$ is a comonad distributive law. Then from Diagram (L1), we get the following identity
\begin{eqnarray*}
&& ((\ddot{\varphi}(n)_\k \o id_B) \ci \ddot{\varphi}(n)_{\k \o B} \ci ((\delta_{\ddot{B}})_\k \o id_H) ) ( (1_\k \o \a^{n+1}_B(b)) \o \a_H^{-1}(h) ) \\
&&~~~~ = ( (\delta_{\ddot{B}})_{\k \o H} \ci \ddot{\varphi}(n)_\k )  ( (1_\k \o \a^{n+1}_B(b)) \o \a_H^{-1}(h) ),
\end{eqnarray*}
which implies (M1) holds.

The other equations can be deduced through a similar way.

$\Leftarrow$: Conversely, for any $(X ,\a_X) \in \overline{\mathcal{H}}^{i,j}(Vec_\Bbbk)$, note that $\ddot{\varphi}(n)_X$
is a morphism in $\overline{\mathcal{H}}^{i,j}(Vec_\Bbbk)$. Then
for any $x \in X$, $b \in B$, $h \in H$, we consider
\begin{eqnarray*}
&&((\ddot{\varphi}(n)_X \o id_B) \ci \ddot{\varphi}(n)_{X \o B} \ci ((\delta_{\ddot{B}})_X \o id_H) ) ( (x \o b) \o h ) \\
&=&(\ddot{\varphi}(n)_X \o id_B)  ( \sum ( (\a_X(x) \o b_1)\o {\a_H(h)}^\varphi ) \o \a_B^n({\a^{-n-2}_B(b_2)}^\varphi ) )\\
&=& \sum ( (\a_X(x) \o {\a_H({\a_H(h)}^\varphi)}^\phi )\o  \a_B^n({\a_B^{-n-1}(b_1)}^\phi) ) \o \a_B^n({\a^{-n-2}_B(b_2)}^\varphi )\\
&=& \sum ( (\a_X(x) \o \a_H({{\a_H(h)}^\varphi}^\phi) )\o  \a_B^{n+1}({\a_B^{-n-2}(b_1)}^\phi) ) \o \a_B^n({\a^{-n-2}_B(b_2)}^\varphi )\\
&\stackrel {(M1)}{=}& \sum ( (\a_X(x) \o \a^2_H({h}^\varphi) )\o  \a_B^{n+1}({{\a_B^{-n-2}(b)}^\varphi}_1) ) \o \a_B^n({{\a^{-n-2}_B(b)}^\varphi}_2 ) \\
&=& \sum ( (\a_X(x) \o \a_H({\a_H(h)}^\varphi) )\o  \a_B^{n}({{\a_B^{-n-1}(b)}^\varphi}_1) ) \o \a_B^{n-1}({{\a^{-n-1}_B(b)}^\varphi}_2 ) \\
&=& (\delta_{\ddot{B}})_{X \o H}  (\sum (x \o {\a_H(h)}^\varphi)  \o (\a_B^n({\a_B^{-n-1}(b)}^\varphi))) \\
&=& ( (\delta_{\ddot{B}})_{X \o H} \ci \ddot{\varphi}(n)_X ) ( (x \o b) \o h ),
\end{eqnarray*}
thus the Diagram (L1) is commute. Similarly to Diagram (L2) - (L4).
\end{proof}

\begin{definition}
Suppose that $(B,\a_B)$ and $(H,\a_H)$ are two Hom-coalgebras over $\k$, $\varphi:B \o H\rightarrow H \o B$ (with notation $\varphi(b \o h) = \sum h^\varphi \o b^\varphi$) is a $\k$-linear map
satisfying $\varphi \ci (\a_B \o \a_H) = (\a_H \o \a_B) \ci \varphi$. If $\varphi$ satisfies Eqs. (M1) - (M4), then we call $\varphi:B \o H\rightarrow H \o B$ a \emph{Hom-cotwistor}.

Further, $B \o H = (B \o H, \a_B \o \a_H,\Delta_{B \o H}, \v_{B \o H})$ is called the \emph{smash coproduct of Hom-coalgebras $B$ and $H$}.
\end{definition}

\begin{theorem}
Let $B,H,\varphi$ and $\ddot{\varphi}(n)$ are defined as above. If we define $\Delta_{B \o H}: B \o H\rightarrow (B \o H) \o (B \o H)$, $\varepsilon_{B \o H}: B \o H\rightarrow \k$ by
$$
\Delta_{B \o H}(b \o h) := \sum (b_1 \o {h_1}^\varphi) \o ( {b_2}^\varphi \o h_2 ),
$$
$$
\v_{B \o H}(b \o h) := \v_B(b)\v_H(h), ~~~~\mbox{where}~~b \in B,~~h \in H,
$$
then the following statements are equivalent:

(1) $\varphi$ is a Hom-cotwistor;

(2) $\ddot{\varphi}(n): \ddot{H}\ddot{B}\rightarrow \ddot{B}\ddot{H}$ is a comonad distributive law;

(3) $\ddot{H}\ddot{B} = ( \_ \o B) \o H$ is a comonad in $\overline{\mathcal{H}}^{i,j}(Vec_\Bbbk)$;

(4) The smash coproduct, $B \o H = (B \o H, \a_B \o \a_H,\Delta_{B \o H}, \v_{B \o H})$, is a Hom-coalgebra over $\k$;

(5) $\_ \o (B \o H)$, the right tensor functor of $B \o H$, is a comonad in $\overline{\mathcal{H}}^{i,j}(Vec_\Bbbk)$.
\end{theorem}

\begin{proof}
(1)$\Leftrightarrow$(2): See Proposition 4.4.

(2)$\Leftrightarrow$(3): See Remark 2.1.

(1)$\Rightarrow$(4): For any $b \in B$, $h \in H$, it is a direct computation to check that
\begin{eqnarray*}
\Delta_{B \o H} (\a_B(b) \o \a_H(h))=(( \a_B \o \a_H ) \o ( \a_B \o \a_H ))(\Delta_{B \o H}(b \o h)),
\end{eqnarray*}
and
\begin{eqnarray*}
\v_{B \o H} (\a_B(b) \o \a_H(h))= \v_B(b) \v_H(h) = \v_{B \o H} (b \o h).
\end{eqnarray*}

Secondly, we have
\begin{eqnarray*}
&&\sum \a_B(b_1) \o \a_H({h_1}^\varphi)  \o   {{b_2}^\varphi}_1 \o {h_{21}}^\phi  \o  {{{b_2}^\varphi}_2}^\phi  \o h_{22}  \\
&\stackrel {(M1)}{=}& \sum \a_B(b_1) \o {{\a_H(h_1)}^\varphi}^\psi  \o   {b_{21}}^\psi \o {h_{21}}^\phi  \o  {{b_{22}}^\varphi}^\phi  \o h_{22} \\
&=& \sum b_{11} \o {{h_{11}}^\varphi}^\psi  \o   {b_{12}}^\psi \o {h_{12}}^\phi  \o  {{\a_B(b_{2})}^\varphi}^\phi  \o \a_H(h_{2}) \\
&\stackrel {(M2)}{=}& \sum b_{11} \o {{{h_{1}}^\varphi}_{1}}^\psi  \o   {b_{12}}^\psi \o {{h_{12}}^\phi}_2  \o  \a_B({b_{2}}^\varphi)  \o \a_H(h_{2}),
\end{eqnarray*}
which implies the Hom-coassociative law. Similarly one can get the Hom-counit law. Thus $B \o H$ is a Hom-coalgebra.

(4)$\Rightarrow$(1): If $B \o H$ is a Hom-coalgebra, then for any $b \in B$, $h \in H$, we have
\begin{eqnarray*}
&& ~~~\sum (\a_B \o \a_H)(b_1\o {h_1}^\varphi) \o \Delta_{B \o H} ( {b_2}^\varphi \o h_2 )\\
&&= \sum \Delta_{B \o H}(b_1\o {h_1}^\varphi) \o (\a_B \o \a_H) ( {b_2}^\varphi \o h_2 ),
\end{eqnarray*}
which implies
\begin{eqnarray*}
&& ~~~\sum (\a_B(b_1) \o \a_H({h_1}^\varphi) ) \o  ( ( {{b_2}^\varphi}_1 \o {h_{21}}^\phi ) \o ( {{{b_2}^\varphi}_2}^\phi  \o h_{22}) ) \\
&&= \sum ( ( b_{11} \o {{{h_1}^\varphi}_1}^\phi ) \o ( {b_{12}}^\phi \o {{h_1}^\varphi}_2) )  \o ( \a_B({b_2}^\varphi) \o \a_H(h_2) ).
\end{eqnarray*}
Use $\v_B \o id_H\o id_B \o \v_H \o id_B \o \v_H$ to action at the above identity, then we get (M1). Use
$\v_B \o id_H\o \v_B \o id_H \o id_B \o \v_H$ to action at the above identity, then we get (M2).

Similarly, one can get (M3) and (M4) through the Hom-counit law of $B \o H$.

(4)$\Leftrightarrow$(5): See Lemma 4.1.
\end{proof}

Note that if $\ddot{\varphi}(n): \ddot{H}\ddot{B}\rightarrow \ddot{B}\ddot{H}$ is a comonad distributive law, $(U,\a_U)$ is an object in the category of $\ddot{\varphi}(n)$-bicomodules, then it means that $(U, \a_U)$ is both $B$-Hom-comodule (with notation $u\mapsto u_{[0]} \o u_{[1]}$) and $H$-Hom-comodule(with notation $u\mapsto u_{(0)} \o u_{(1)}$),
and satisfies the following identity
$$
{u_{[0]}}_{(0)} \o {u_{[0]}}_{(1)} \o u_{[1]} = {u_{(0)}}_{[0]} \o {\a_H(u_{(1)})}^\varphi \o \a_B^n({\a_B^{-n-1}({u_{(0)}}_{[1]})}^\varphi).
$$
We write the category of $\ddot{\varphi}(n)$-bicomodules by $\overline{\mathcal{H}}^{i,j}(Vec_\Bbbk)^{(\ddot{H},\ddot{B})}(\ddot{\varphi}(n))$, and write the category of $B \o H$-Hom-comodules by $Corep(B \o H)$,
then we have the following property.

\begin{proposition}
For all integer $n$, $\overline{\mathcal{H}}^{i,j}(Vec_\Bbbk)^{(\ddot{H},\ddot{B})}(\ddot{\varphi}(n))$ is isomorphic to $Corep(B \o H)$.
\end{proposition}

\begin{proof}
We define a functor $P_n:Corep(B \o H) \rightarrow \overline{\mathcal{H}}^{i,j}(Vec_\Bbbk)^{(\ddot{H},\ddot{B})}(\ddot{\varphi}(n))$ as follows:

$\bullet$ for any object $(U,\a_U) \in Corep(B \o H)$ (for any $u \in U$, the coaction $U\rightarrow U \o (B \o H)$ is written by $u\mapsto \sum u_{<0>} \o (u_{<1>B} \o u_{<1>H})$),
$P_n(U,\a_U):=(U,\a_U)$, the coaction $\theta^U:U\rightarrow U \o H$ and $\rho^U:U\rightarrow U \o B$ are given by:
\begin{eqnarray*}
&&\theta^U(u) = u_{(0)} \o u_{(1)} := \sum u_{<0>} \o \v_B(u_{<1>B}) \a_H^{-1}(u_{<1>H}),\\
&&\rho^U(u) = u_{[0]} \o u_{[1]} := \sum u_{<0>} \o \a_B^{n}(u_{<1>B}) \v_H(u_{<1>H});
\end{eqnarray*}

$\bullet$ for any morphism $\varpi: (U,\a_U)\rightarrow (V,\a_V)$, $P_n(\varpi):=\varpi$.

Firstly, we have $\v_H(u_{(1)})u_{(0)} = \v_{B \o H}(u_{<1>}) u_{<0>} = \a_U(u)$, and
$\theta^U(\a_U(u)) = \a_U(u_{(0)}) \o \a_H(u_{(1)})$.
Further, for any $u \in U$, we compute
\begin{eqnarray*}
&&~~~\a_U(u_{(0)}) \o \Delta_H(u_{(1)}) \\
&&= \sum \a_U(u_{<0>}) \o \v_B(u_{<1>B}) \a_H^{-1}(u_{<1>H1}) \o \a_H^{-1}(u_{<1>H2}) \\
&&= \sum \a_U(u_{<0>}) \o \v_B(u_{<1>B1}) \a_H^{-1}({u_{<1>H1}}^\varphi) \v_B({u_{<1>B2}}^\varphi) \o \a_H^{-1}(u_{<1>H2})\\
&&= \sum \a_U(u_{<0>}) \o (\v_B \o \a_H^{-1})(u_{<1>B} \o u_{<1>H})_1 \o (\v_B \o \a_H^{-1})(u_{<1>B} \o \a_H^{-1}(u_{<1>H})_2) \\
&&= \sum u_{<0><0>} \o \v_B({u_{<0>}}_{<1>B}) \a_H^{-1}({u_{<0>}}_{<1>H}) \o \v_B(u_{<1>B}) u_{<1>H} \\
&&= \theta^U(u_{(0)}) \o \a_H(u_{(1)}).
\end{eqnarray*}
Thus $(U, \a_U, \theta^U)$ is a right $H$-Hom-comodule.

Secondly, we can show that $(U, \a_U, \rho^U)$ is a right $B$-Hom-comodule similarly.

Thirdly, we have
\begin{eqnarray*}
&&~~~{u_{(0)}}_{[0]} \o {\a_H(u_{(1)})}^\varphi \o \a_B^n({\a_B^{-n-1}({u_{(0)}}_{[1]})}^\varphi) \\
&&= \sum u_{<0><0>} \o \v_B(u_{<1>B}) (u_{<1>H})^\varphi \o \v_H(u_{<0><1>H} ) \a_B^n({\a_B^{-1}( u_{<0><1>B} )}^\varphi) \\
&&= \sum \a_U(u_{<0>}) \o \v_B({u_{<1>B2}}^\phi) {\a_H^{-1}(u_{<1>H2})}^\varphi \o \v_H({u_{<1>H1}}^\phi) \a_B^n({\a_B^{-1}( u_{<1>B1} )}^\varphi) \\
&&= \sum \a_U(u_{<0>}) \o \v_B(u_{<1>B1})\a_H^{-1}({u_{<1>H1}}^\varphi) \o \v_H(u_{<1>H2}) \a_B^{n-1}({ u_{<1>B2} }^\varphi) \\
&&= {u_{[0]}}_{(0)} \o {u_{[0]}}_{(1)} \o u_{[1]}.
\end{eqnarray*}

At last, it is easily to show that $P(\varpi)$ is both $B$-colinear and $H$-colinear for any morphism $\varpi: (U,\a_U)\rightarrow (V,\a_V)$ in $Corep(B \o H)$.
Thus $P_n$ is well defined.

Conversely, Define a functor $Q_n:\overline{\mathcal{H}}^{i,j}(Vec_\Bbbk)^{(\ddot{H},\ddot{B})}(\ddot{\varphi}(n)) \rightarrow Corep(B \o H)$ as follows:

$\bullet$ for any object $(U,\a_U) \in \overline{\mathcal{H}}^{i,j}(Vec_\Bbbk)^{(\ddot{H},\ddot{B})}(\ddot{\varphi}(n))$,
$Q_n(U,\a_U):=(U,\a_U)$, the coaction $\varsigma^U:U\rightarrow U \o (B \o H)$ is given by:
\begin{eqnarray*}
\varsigma^U(u) = u_{<0>} \o u_{<1>} := \sum \a^{-1}_U({u_{(0)}}_{[0]}) \o ( \a_B^{-n-1}({u_{(0)}}_{[1]}) \o \a_H^{-1}(u_{(1)}) );
\end{eqnarray*}

$\bullet$ for any morphism $\varpi: (U,\a_U)\rightarrow (V,\a_V)$, $Q_n(\varpi):=\varpi$.

It is straightforward to check that $Q_n$ is well defined, and $P_n$ is the inverse of $Q_n$.
\end{proof}

Now assume that $(B,\a_B)$ and $(H,\a_H)$ are all Hom-bialgebras over $\k$, $\varphi:B \o H\rightarrow H \o B$ a
$\k$-linear map, $\ddot{\varphi}(n): \ddot{H}\ddot{B}\rightarrow \ddot{B}\ddot{H}$ is defined as above. If we define
$\ddot{B}_2$, $\ddot{B}_0$ and $\ddot{H}_2$, $\ddot{H}_0$ as in Lemma \ref{mdl}, then we have the following conclusion.

\begin{proposition}
$\ddot{\varphi}(n): \ddot{H}\ddot{B}\rightarrow \ddot{B}\ddot{H}$ is a monoidal comonad distributive law on $\overline{\mathcal{H}}^{i,j}(Vec_\Bbbk)$ for any $n \in \mathbb{Z}$ if and only if for any $a,b \in B$, $h,g \in H$,
$\varphi$ is a Hom-cotwistor and satisfies
\\
$\left\{\begin{array}{l}
(M5)~ \sum  h^\varphi g^\phi \o a^\varphi b^\phi = \sum (hg)^\varphi \o (ab)^\varphi;\\
(M6)~ \sum {1_H}^\varphi \o {1_B}^\varphi = 1_H \o 1_B.
\end{array}\right.$
\end{proposition}

\begin{proof}
$\Rightarrow$: We only need to check (M5) and (M6). For any $a,b \in B$, $h,g \in H$, note that
\begin{eqnarray*}
&& ~~~ \varphi_{\k \o \k} \ci (\ddot{B}_2(\k,\k) \o id_H) \ci \ddot{H}_2(\k \o B,\k \o B)\\
&& = (\ddot{H}_2(\k,\k) \o id_B) \ci \ddot{B}_2(\k \o H,\k \o H) \ci (\varphi_\k \o \varphi_\k).
\end{eqnarray*}
Take the above identity to action at $((1_\k \o \a_B^{-i-n-1}(a)) \o \a_H^{-i-1}(h)) \o ((1_\k \o \a_B^{-j-n-1}(b)) \o \a_H^{-j-1}(g))$, then we immediately get
\begin{eqnarray*}
&& ~~~\sum  \a_H^{i}({\a_H^{-i}(h)}^\varphi)\a_H^{j}({\a_H^{-j}(g)}^\phi) \o \a_B^{n+i}( {a_B^{-i}(a)}^\varphi ) \a_B^{n+j} ( {\a_B^{-j}(b)}^\phi ) \\
&& =\sum  {( hg )}^\varphi \o \a_B^n({(ab)}^\varphi),
\end{eqnarray*}
which implies (M5). Similarly, the equation b) in Lemma 3.1 implies (M6).

$\Leftarrow$:
Since Proposition 4.4, $\ddot{\varphi}(n)$ is a comonad distributive law on $\overline{\mathcal{H}}^{i,j}(Vec_\Bbbk)$.

For any $(X ,\a_X), (Y, \a_Y) \in \overline{\mathcal{H}}^{i,j}(Vec_\Bbbk)$, $x \in X$, $y \in Y$, $a,b \in B$, $h,g \in H$, we compute
\begin{eqnarray*}
&& (\varphi_{X \o Y} \ci (\ddot{B}_2(X,Y) \o id_H) \ci \ddot{H}_2(X \o B,Y \o B)) ( (( x \o a) \o h) \o ((y \o b) \o g) )  \\
&=& \varphi_{X \o Y} ( (x \o y) \o \a_B^i(a)\a_B^j(b) \o \a_H^i(h)\a_H^j(g) ) \\
&=& \sum ( (x \o y) \o {\a_H(\a_H^i(h)\a_H^j(g))}^\varphi ) \o \a_B^n( {\a_B^{-n-1}(\a_B^i(a)\a_B^j(b))}^\varphi )\\
&\stackrel {(M5)}{=}& \sum ( (x \o y) \o {\a_H^{i+1}(h)}^\varphi {\a_H^{j+1}(g)}^\phi ) \o \a_B^n( {\a_B^{i-n-1}(a)}^\varphi {\a_B^{j-n-1}(b)}^\phi )\\
&=& \sum ( (x \o y) \o \a_H^i({\a_H(h)}^\varphi) \a_H^{j}({\a_H(g)}^\phi) ) \o \a_B^{n+i}( {\a_B^{-n-1}(a)}^\varphi ) \a_B^{n+j}({\a_B^{-n-1}(b)}^\phi )\\
&=& \ddot{H}_2(X \o B,Y \o B) (\sum (( x \o {\a_H(h)}^\varphi ) \o (y \o {\a_H(g)}^\phi) ) \\
&&~~~\o \a_B^i(\a_B^n( {\a_B^{-n-1}(a)}^\varphi )) \a_B^j(\a_B^n( {\a_B^{-n-1}(b)}^\phi )) )\\
&=& (\varphi_{X \o Y} \ci (\ddot{B}_2(X,Y) \o id_H) \ci \ddot{H}_2(X \o B,Y \o B) ) ((( x \o a) \o h) \o ((y \o b) \o g)),
\end{eqnarray*}
thus condition a) in Lemma 3.1 holds. Samely, we have condition b). Thus $\ddot{\varphi}(n)$ is a monoidal comonad distributive law.
\end{proof}

\begin{definition}
If the Hom-cotwistor $\varphi:B \o H\rightarrow H \o B$ satisfies Eqs. (M5) - (M6), then we call $\varphi$ a
\emph{monoidal Hom-cotwistor}.
\end{definition}

\begin{theorem}
Assume that $(B,\a_B)$ and $(H,\a_H)$ are two Hom-bialgebras over $\k$, $\varphi:B \o H\rightarrow H \o B$ (with notation $\varphi(b \o h) = \sum h^\varphi \o b^\varphi$) is a $\k$-linear map
satisfying $\varphi \ci (\a_B \o \a_H) = (\a_H \o \a_B) \ci \varphi$. For any $(X, \a_X) \in \overline{\mathcal{H}}^{i,j}(Vec_\Bbbk)$, if we define the natural transformation $\ddot{\varphi}(n)_X: \ddot{H}\ddot{B}(X)\rightarrow \ddot{B}\ddot{H}(X)$ by
\begin{eqnarray*}
\ddot{\varphi}(n)_X: (X \o B) \o H\rightarrow (X \o H) \o B,~~~~(x \o b) \o h\mapsto \sum (x \o {\a_H(h)}^\varphi) \o \a_B^{n}({\a_B^{-n-1}(b)}^\varphi),
\end{eqnarray*}
and define $\Delta_{B \o H}: B \o H\rightarrow (B \o H) \o (B \o H)$, $\varepsilon_{B \o H}: B \o H\rightarrow \k$ by
$$
\Delta_{B \o H}(b \o h) := \sum (b_1 \o {h_1}^\varphi) \o ( {b_2}^\varphi \o h_2 ),
$$
$$
\v_{B \o H}(b \o h) := \v_B(b)\v_H(h), ~~~~\mbox{where}~~b \in B,~~h \in H,
$$
then the following statements are equivalent:

(1) Define the following monoidal structure in $\overline{\mathcal{H}}^{i,j}(Vec_\Bbbk)^{(\ddot{H},\ddot{B})}(\ddot{\varphi}(n))$:
\begin{eqnarray*}
&&\theta^{U \o V}:U \o V\rightarrow (U \o V) \o H, ~~~~u \o v\mapsto (u_{(0)} \o v_{(0)}) \o \a_H^i(u_{(1)})\a_H^j(v_{(1)}),\\
&&\rho^{U \o V}:U \o V\rightarrow (U \o V) \o B, ~~~~u \o v\mapsto (u_{[0]} \o v_{[0]}) \o \a_H^i(u_{[1]})\a_H^j(v_{[1]}),
\end{eqnarray*}
where $U,V$ are all objects in $\overline{\mathcal{H}}^{i,j}(Vec_\Bbbk)^{(\ddot{H},\ddot{B})}(\ddot{\varphi}(n))$, $u \in U$, $v \in V$. The associativity constraint and the unit constraint are same as $\overline{\mathcal{H}}^{i,j}(Vec_\Bbbk)$. Then $\overline{\mathcal{H}}^{i,j}(Vec_\Bbbk)^{(\ddot{H},\ddot{B})}(\ddot{\varphi}(n))$ is a monoidal category.

(2) $\varphi$ is a monoidal Hom-cotwistor;

(3) $\ddot{\varphi}(n): \ddot{H}\ddot{B}\rightarrow \ddot{B}\ddot{H}$ is a monoidal comonad distributive law;

(4) $\ddot{H}\ddot{B} = ( \_ \o B) \o H$ is a bicomonad in $\overline{\mathcal{H}}^{i,j}(Vec_\Bbbk)$;

(5) $B \o H = (B \o H, \a_B \o \a_H,\Delta_{B \o H}, \v_{B \o H}, m_{B \o H}, 1_B \o 1_H)$ is a Hom-bialgebra over $\k$, where $m_{B \o H}( (a \o h)  \o (b \o g)):= ab \o gh$;

(6) $\_ \o (B \o H)$, the right tensor functor of $B \o H$, is a bicomonad in $\overline{\mathcal{H}}^{i,j}(Vec_\Bbbk)$;

(7) Deine the monoidal structure in $Corep^{i,j}(B \o H)$:
\begin{eqnarray*}
\theta^{U \o V}(u \o v):= u_{<0>} \o v_{<0>} \o \a_{B \o H}^i(u_{<1>})\a_{B \o H}^j(v_{<1>}),
\end{eqnarray*}
where $U,V$ are all $B \o H$-Hom-comodules, $u \in U$, $v \in V$, and define the associativity constraint and the unit constraint in $Corep^{i,j}(B \o H)$ be same as $\overline{\mathcal{H}}^{i,j}(Vec_\Bbbk)$, then $Corep^{i,j}(B \o H)$ is a monoidal category. Moreover, for any integer $n$, $\overline{\mathcal{H}}^{i,j}(Vec_\Bbbk)^{(\ddot{H},\ddot{B})}(\ddot{\varphi}(n))$ is monoidal isomorphic to $Corep^{i,j}(B \o H)$.
\end{theorem}

\begin{proof}
(1)$\Leftrightarrow$(3): See Theorem 3.5.

(2)$\Leftrightarrow$(3): See Theorem 4.6 and Proposition 4.8.

(3)$\Leftrightarrow$(4): See Lemma 3.4.

(2)$\Rightarrow$(5): Since Theorem 4.6, $B \o H$ is a Hom-coalgebra over $\k$. Obviously $B \o H$ is also a Hom-algebra. We only need to check that $\Delta_{B \o H}$ and $\v_{B \o H}$ are Hom-algebra morphisms.

For any $a,b \in B$, $g,h \in H$, it is a direct computation to check that
\begin{eqnarray*}
&& \Delta_{B \o H}( a \o h )  \Delta_{B \o H}( b \o g )\\
&=& \sum ( (a_1 \o {h_1}^\varphi) \o ( {a_2}^\varphi \o h_2 ) ) ( (b_1 \o {g_1}^\varphi) \o ( {b_2}^\varphi \o g_2 ) )  \\
&=& \sum  ( a_1 b_1 \o {h_1}^\varphi{g_1}^\varphi ) \o ( {a_2}^\varphi{b_2}^\varphi \o h_2 g_2 ) \\
&\stackrel {(M5)}{=}& \sum ( {(a b)}_1 \o {{(hg)}_1}^\varphi ) \o ( {{(ab)}_2}^\varphi \o {(hg)}_2 ) \\
&=& \Delta_{B \o H}(ab \o gh)  ,
\end{eqnarray*}
and
\begin{eqnarray*}
\Delta_{B \o H}( 1_H \o 1_H )&=& \sum 1_B \o {1_H}^\varphi \o {1_B}^\varphi \o 1_H\\
&\stackrel {(M6)}{=}& \sum 1_B \o 1_H \o 1_B \o 1_H,
\end{eqnarray*}
thus $\Delta_{B \o H}$ is a Hom-algebra morphism. Similarly, $\v_{B \o H}$ is also a Hom-algebra morphism.

(5)$\Rightarrow$(2): We only need to prove (M5) and (M6). For any $a,b \in B$, $h,g \in H$, since $B \o H$ is a Hom-bialgebra, we immediately get that
\begin{eqnarray*}
&\sum  ( \a^{-1}_B(a_1) \a^{-1}_B(b_1) \o {\a^{-1}_H(h_1)}^\varphi{\a^{-1}_H(g_1)}^\varphi ) \o ( {\a^{-1}_B(a_2)}^\varphi{\a^{-1}_B(b_2)}^\varphi \o \a^{-1}_H(h_2) \a^{-1}_H(g_2) ) \\
&= \sum ( {(\a^{-1}_B(a) \a^{-1}_B(b))}_1 \o {{(\a^{-1}_H(h) \a^{-1}_H(g))}_1}^\varphi ) \o ( {{(\a^{-1}_B(a) \a^{-1}_B(b))}_2}^\varphi \o {(\a^{-1}_H(h)\a^{-1}_H(g))}_2 ).
\end{eqnarray*}
Use $\v_B \o id_H \o id_B \o \v_H$ to act at the both side of the above identity, we obtain (M5). Similarly one can get (M6).

(5)$\Leftrightarrow$(6): See Theorem 4.3.

(6)$\Leftrightarrow$(7): See Theorem 4.3.
\end{proof}

\section{\textbf{Hom-entwining structures}}
\def\theequation{5.\arabic{equation}}
\setcounter{equation} {0} \hskip\parindent

Suppose that $(A,\a_A)$ is a Hom-algebra over $\k$. Then $(A^{\ast cop}, \a_{A^{\ast cop}})$ is a Hom-coalgebra with the following structures:
\begin{eqnarray*}
&\overline{\v}(f):=f(1_A),~~\overline{\Delta}(f)(xy):=f(\a_A^{-2}(y x)),~~\a_{A^{\ast cop}}(f):=f\ci \a_A^{-1},
\end{eqnarray*}
where $x,y \in A$, $f \in A^{\ast cop}$.

Assume that $(H,\a_H)$ is also a Hom-coalgebra over $\k$. Then we have the following property.

\begin{theorem}
If $\varphi:{A^\ast}^{cop} \o H\rightarrow H \o {A^\ast}^{cop}$ is a $\k$-linear map, then $\varphi$ is a Hom-cotwistor if and only if there is a $\k$-linear map $\Phi:H \o A\rightarrow A \o H$ (with notation $\Phi(h \o a) = \sum a_\Phi \o h^\Phi$), satisfying
$(\a_A \o \a_H) \ci \Phi = \Phi \ci (\a_H \o \a_A)$ and the following identities:
\\
$\left\{\begin{array}{l}
(E1)~ \sum a_\Phi b_\Psi \o \a_H({h}^{\Phi\Psi}) = \sum{(ab)}_\Phi \o {\a_H(h)}^\Phi;\\
(E2)~ \sum {\a_A(a)}_{\Psi\Phi} \o {h_1}^\Phi \o {h_2}^\Psi = \sum \a_A(a_\Phi) \o {h^\Phi}_1 \o {h^\Phi}_2;\\
(E3)~ \sum \v_H(h^\Phi) a_\Phi = \v_H(h) a;\\
(E4)~ \sum {(1_A)}_\Phi \o h^\Phi = 1_A \o h,
\end{array}\right.$\\
where $\Phi = \Psi$, $a,b \in A$, $h \in H$.
\end{theorem}

\begin{proof}
Let $\Pi = \{ \Phi:H \o A\rightarrow A \o H \mbox{~is~a~linear~map}\mid \Phi \mbox{~satisfies~Eqs.(E1)-(E4),~and~}(\a_A \o \a_H) \ci \Phi = \Phi \ci (\a_H \o \a_A) \}$, and $\Xi$ denote the collection of the Hom-cotwistors $\varphi:{A^\ast}^{cop} \o H\rightarrow H \o {A^\ast}^{cop}$.

For any $\varphi \in \Xi$, define a map $\varrho: \Xi\rightarrow \Pi$ by $\varrho(\varphi) = \Phi$, where $\Phi$ is defined as follows
$$
\Phi(h \o a) = \sum a_\Phi \o h^\Phi := \sum {e^i}^\varphi(a)e_i \o h^\varphi,
$$
where $a \in A$, $h \in H$, $e_i$ and $e^i$ are dual bases of $A$ and $A^\ast$ respectively.

For any $\Phi \in \Pi$, define a map $\sigma:\Pi\rightarrow \Xi$ by $\sigma(\Phi) = \varphi$, where $\varphi$ is defined as follows
$$
\varphi(f \o h) = \sum h^\varphi \o f^\varphi := \sum f({e_i}_\Phi) h^\Phi \o e^i,
$$
where $f \in A^\ast$, $h \in H$, $e_i$ and $e^i$ are dual bases of $A$ and $A^\ast$ respectively.

Obviously $\varrho$ is the inverse of $\sigma$. From now on, assume that $\varphi \in \Xi$ and $\Phi \in \Pi$ are in correspondence with each other.

(1) If $\varphi \ci (\a_{{A^\ast}^{cop}} \o \a_H) = (\a_H \o \a_{{A^\ast}^{cop}}) \ci \varphi$, then for any $f\in A^\ast$, we have
\begin{eqnarray*}
\sum f(\a_A(a_\Phi)) \o \a_H(h^\Phi) &=& \sum {e^i}^\varphi(a)f(\a_A(e_i)) \o \a_H(h^\varphi)\\
&=& \sum {\a^{-1}_{{A^\ast}^{cop}}(f)}^\varphi (a) \o \a_H(h^\varphi)\\
&=& \sum \a^{-1}_{{A^\ast}^{cop}}(f^\varphi)(a) \o {\a_H(h)}^\varphi\\
&=& \sum {e^i}^\varphi(\a_A(a))f(e_i) \o {\a_H(h)}^\varphi = \sum f({\a_A(a)}_\Phi) \o \a_H(h)^\Phi,
\end{eqnarray*}
which implies $(\a_A \o \a_H) \ci \Phi = \Phi \ci (\a_H \o \a_A)$.

Conversely, if $(\a_A \o \a_H) \ci \Phi = \Phi \ci (\a_H \o \a_A)$, then for any $a \in A$, we compute
\begin{eqnarray*}
\sum \a_H(h^\varphi) \o \a_{{A^\ast}^{cop}}(f^\varphi) &=& \sum f({e_i}_\Phi) \a_H(h^\Phi) \o \a_{{A^\ast}^{cop}}(e^i) \\
&=& \sum f({\a_A^{-1}(e_i)}_\Phi) \a_H(h^\Phi) \o \ e^i \\
&=& \sum f(\a_A^{-1}({e_i}_\Phi)) {\a_H(h)}^\Phi \o \ e^i \\
&=& \sum \a_{{A^\ast}^{cop}}(f)({e_i}_\Phi){\a_H(h)}^\Phi \o \ e^i = \sum {\a_H(h)}^\varphi \o {\a_{{A^\ast}^{cop}}(f)}^\varphi.
\end{eqnarray*}

(2) Moreover, if (M1) holds, then for any $f\in A^\ast$, $a,b \in A$, $h \in H$, we immediately get
\begin{eqnarray*}
\sum {\a_H(h)}^{\varphi\phi} \o {f_1}^\phi(\a_A(b)) \o {f_2}^\varphi(\a_A(a)) = \sum  \a_H(h^\varphi) \o {f^\varphi}_1(\a_A(b)) \o {f^\varphi}_2(\a_A(a)).
\end{eqnarray*}

For one thing, we have
\begin{eqnarray*}
&&~~~\sum {\a_H(h)}^{\varphi\phi} \o {f_1}^\phi(\a_A(b)) \o {f_2}^\varphi(\a_A(a)) \\
&&=\sum {\a_H(h)}^{\Phi\Psi} \o f( \a^{-2}_A ({\a_A(a)}_\Phi) \a^{-2}_A ({\a_A(b)}_\Psi) ) \o 1_\k\\
&&= \sum \a_H(h^{\Phi\Psi}) \o f( \a^{-1}_A (a_\Phi) \a^{-1}_A (b_\Psi) ) \o 1_\k\\
&&= \sum \a_H(h^{\Phi\Psi}) \o \a_{{A^\ast}^{cop}}(f)(a_\Phi b_\Psi) \o 1_\k.
\end{eqnarray*}

For another, we have
\begin{eqnarray*}
\sum \a_H(h^\varphi) \o {f^\varphi}_1(\a_A(b)) \o {f^\varphi}_2(\a_A(a)) &=& \sum \a_H(h^\Phi) \o f( {\a^{-1}_A(ab)}_\Phi ) \o 1_\k\\
&=& \sum {\a_H(h)}^\Phi \o f( \a^{-1}_A({(ab)}_\Phi) ) \o 1_\k\\
&=& \sum \a_H(h)^\Phi \o \a_{{A^\ast}^{cop}}(f)( {(ab)}_\Phi ) \o 1_\k.
\end{eqnarray*}
Thus Eq.(E1) holds.

Conversely, if (E1) holds, we can show Eq.(M1).

(3) Be similar with (2), Eq.(E2) and (M2) can be deduced from each other.

(4) If (M3) holds, then for any $f\in A^\ast$, $a \in A$, $h \in H$, we have
\begin{eqnarray*}
\sum \v_H(h^\varphi) f^\varphi(a) = \v_H(h) f(a),
\end{eqnarray*}
which implies
\begin{eqnarray*}
\sum \v_H(h^\Phi) f(a_\Phi) = \v_H(h) f(a),
\end{eqnarray*}
thus (E3) holds.

Conversely, (E3) also implies (M3).

(5) Be similar with (4), Eq.(E4) and (M4) can be deduced from each other.
\end{proof}

\begin{definition}
Suppose that $(A,\a_A)$ is a Hom-algebra, $(H,\a_H)$ is a Hom-coalgebra over $\k$. If there is a $\k$-linear map $\Phi:H \o A\rightarrow A \o H$ (with notation $\Phi(h \o a) = \sum a_\Phi \o h^\Phi$), satisfying
$(\a_A \o \a_H) \ci \Phi = \Phi \ci (\a_H \o \a_A)$ and Eqs.(E1)-(E4), then we call $\Phi$ a \emph{Hom-entwining map}, and call the triple $(H,A,\Phi)$ a \emph{(right-right) Hom-entwining structure over $\k$}.
\end{definition}

\begin{definition}
Assume that $n$ is an integer in $\mathbb{Z}$, $(H,A,\Phi)$ a (right-right) Hom-entwining structure, $(U,\a_U)$ is both a right $A$-Hom-module and a right $H$-Hom-comodule. If the following
identity
\begin{eqnarray}
\rho^U(u \c a) = \sum u_{(0)} \c\a_A({a}_\Phi) \o \a_H^{-n}({\a_H^{n+1}(u_{(1)})}^\Phi)
\end{eqnarray}
holds for all $a \in A$, $h \in H$, $u \in U$, then we call $(U,\a_U)$ an \emph{$n$-th Hom-entwined module} or an \emph{$n$-th $\Phi$-module}. The category of $n$-th Hom-entwined modules and $A$-linear $H$-colinear maps is denoted by $\mathcal{M}(\Phi)^H_A(n)$.
\end{definition}

\begin{example}
(1) If $(H,A,\Phi)$ is a Hom-entwining structure over $\k$, then $H\o A$ is an $n$-th Hom-entwined module under the following structures
\begin{eqnarray*}
&\rho^{H \o A}(h \o a) = \sum (h_1 \o \a_A^{n+1}({\a_A^{-n}(a)}_\Phi)) \o {h_2}^\Phi,\\
&(h \o a) \c x = \a_H(h) \o a\a_A^{-1}(x),
\end{eqnarray*}
where $a,x \in A$, $h \in H$.

\begin{proof}
It is easy to get that $H \o A$ is both an $A$-Hom-module and an $H$-Hom-comodule. We only check Eq. (5.1).
For any $a,x \in A$, $h \in H$, we compute as follows:
\begin{eqnarray*}
\rho^{H \o A}((h \o a) \c x) &=& \rho^{H \o A} (\a_H(h) \o a\a_A^{-1}(x))\\
&=& \sum (\a_H(h_1) \o \a_A^{n+1}({(\a_A^{-n}(a)\a_A^{-n-1}(x))}_\Phi)) \o {\a_H(h_2)}^\Phi\\
&\stackrel {E1}{=}& \sum (\a_H(h_1) \o \a_A^{n+1}({\a_A^{-n}(a)}_\Phi) \a_A^{n+1}({\a_A^{-n-1}(x)}_\Psi) ) \o \a_H({h_2}^{\Phi\Psi})\\
&=& \sum (\a_H(h_1) \o \a_A^{n+1}({\a_A^{-n}(a)}_\Phi) {x}_\Psi ) \o \a_H^{-n}(\a_H^{n+1}({h_2}^{\Phi})^\Psi) \\
&=& \sum (h_1 \o \a_A^{n+1}({\a_A^{-n}(a)}_\Phi) ) \c \a_A(x_\Psi)\o \a_H^{-n}(\a_H^{n+1}({h_2}^{\Phi})^\Psi) \\
&=& \sum (h \o a)_{(0)} \c \a_A(x_\Psi) \o \a_H^{-n}(\a_H^{n+1}({(h \o a)}_{(1)})^\Psi).
\end{eqnarray*}
Thus $H \o A$ is an object in $\mathcal{M}(\Phi)^H_A(n)$.
\end{proof}

(2) Similarly, if $(H,A,\Phi)$ is a Hom-entwining structure over $\k$, then $A\o H$ is an $n$-th Hom-entwined module under the following structures
\begin{eqnarray*}
&\rho^{A \o H}(a \o h) = (\a_A(a) \o h_1) \o \a_H^{-n}(h_2),\\
&(a \o h) \c x = \sum a ({\a_A(x)}_\Phi) \o {\a_H(h)}^\Phi,
\end{eqnarray*}
where $a,x \in A$, $h \in H$.

(3) If all the Hom-structure map $\a=id$, then the Hom-entwining structure is the usual entwining structure, and the $n$-th Hom-entwined modules are usual entwined modules.

(4) Assume that $(H,\a_H)$ is a Hom-bialgebra over $\k$, $n$ is an integer. If we define $\Phi$ by
\begin{equation*}
\begin{split}
\Phi: H\otimes H  &\longrightarrow H\otimes H \\
h\otimes g &\longmapsto \a_H^{-1}(g_{1}) \otimes \a_H^{-1}(h) \a_H^{n}(g_{2}),
\end{split}
\end{equation*}
where $h,g \in H$, then it is a direct computation to check that $(H,H,\Phi)$ is a Hom-entwining structure.
Further, $(U,\a_U)$ is an object in $\mathcal{M}(\Phi)^H_H(n)$ means that
$$
\rho(u \c h) = u_{[0]}\c h_{1} \o u_{[1]} h_{2},~~u \in U,~~h \in H,
$$
which implies the $n$-th Hom-entwined modules are actually Hom-Hopf modules over $H$.

(5) If we take $(A,\a_A)=(\k,id_\k)$, and $\Phi:H \o \k \rightarrow \k \o H$ is defined as $id_H$, then it is easy to get that $\Phi$ is a Hom-entwining map. Further,
note that if $(U,\a_U,\leftharpoonup)$ is a right $\k$-Hom-module, then we have
$$
u \leftharpoonup \lambda = \lambda \a_U(u)
$$
for any $u \in U$ and $\lambda \in \k$. Thus for any $n \in \mathbb{Z}$, $(U,\a_U)$ is an object in $\mathcal{M}(\Phi)^H_\k(n)$ means that
$$
\rho(u \c \lambda) = u_{[0]}\c \lambda \o \a_H(u_{[1]}),
$$
which implies $\mathcal{M}(\Phi)^H_\k(n) = \mathcal{M}^H$.

(6) Similarly, if we take $(H,\a_H)=(\k,id_\k)$, and $\Phi: \k \o A \rightarrow A \o \k $ is defined as $id_A$, then it is easy to get that $\Phi$ is a Hom-entwining map, and $\mathcal{M}(\Phi)^\k_A(n)$ is identified to $\mathcal{M}_A$ for any $n \in \mathbb{Z}$.

\end{example}

Recall from Theorem 5.1, if $\Phi:H \o A\rightarrow A \o H$ is a Hom-entwining map, then there is a Hom-cotwistor $\varphi:{A^\ast}^{cop} \o H\rightarrow H \o {A^\ast}^{cop}$. Thus for $n \in \mathbb{Z}$, we can get a comonad distributive law $\ddot{\varphi}(n): \ddot{H}\ci {\ddot{A}}^{\ast cop}\rightarrow {\ddot{A}}^{\ast cop} \ci \ddot{H}$.

\begin{proposition}
$\mathcal{M}(\Phi)^H_A(n)$ and $\overline{\mathcal{H}}^{i,j}(Vec_\Bbbk)^{(\ddot{H},{\ddot{A}}^{\ast cop})}(\ddot{\varphi}(n))$ are isomorphic.
\end{proposition}

\begin{proof}
We define a functor $R_n: \mathcal{M}(\Phi)^H_A(n) \rightarrow \overline{\mathcal{H}}^{i,j}(Vec_\Bbbk)^{(\ddot{H},{\ddot{A}}^{\ast cop})}(\ddot{\varphi}(n))$ as follows:

$\bullet$ for any object $(U,\a_U) \in \mathcal{M}(\Phi)^H_A(n)$ (suppose that the $H$-coaction is given by $u\mapsto u_{(0)} \o u_{(1)}$, and the $A$-action is given by $u \o a \mapsto u \c a$),  $R_n(U):=U$ as a $\k$-module, and the $\ddot{H}$-comodule structure is also defined by $u\mapsto u_{(0)} \o u_{(1)}$, the ${\ddot{A}}^{\ast cop}$-comodule structure is defined by
$$
u\mapsto u_{[0]} \o u_{[1]} := \sum m \c e_i \o e^i,
$$
where $ u \in U$, $e_i$ and $e^i$ are dual bases of $A$ and $A^\ast$ respectively;

$\bullet$ for any morphism $\varpi: (U,\a_U)\rightarrow (V,\a_V)$, $R_n(\varpi):=\varpi$.

Obviously $U$ is an ${\ddot{A}}^{\ast cop}$-comodule under the above structure. Then we compute
\begin{eqnarray*}
&&~~~~ \sum {u_{(0)}}_{[0]} \o {\a_H(u_{(1)})}^\varphi \o \a_{{A^\ast}^{cop}}^n({\a_{{A^\ast}^{cop}}^{-n-1} ({u_{(0)}}_{[1]})}^\varphi )(a) \\
&&= \sum {u_{(0)}}_{[0]} \o {\a_H(u_{(1)})}^\varphi \o ({\a_{{A^\ast}^{cop}}^{-n-1} ({u_{(0)}}_{[1]})}^\varphi )(\a_A^{-n}(a)) \\
&&= \sum {u_{(0)}}_{[0]} \o {u_{(0)}}_{[1]} (\a_A^{n+1} ({\a_A^{-n}(a)}_\Phi) ) {\a_H(u_{(1)})}^\Phi \o 1_\k \\
&&= \sum u_{(0)} \c \a_A^{n+1} ({\a_A^{-n}(a)}_\Phi) \o {\a_H(u_{(1)})}^\Phi \o 1_\k \\
&&= \sum u_{(0)} \c \a_A(a_\Phi) \o \a_A^{-n} ({\a_H^{n+1}(u_{(1)})}^\Phi) \o 1_\k \\
&&= \sum (u \c a)_{(0)} \o (u \c a)_{(1)} \o 1_\k  = {u_{[0]}}_{(0)} \o {u_{[0]}}_{(1)} \o u_{[1]} (a),
\end{eqnarray*}
hence $R_n$ is well defined.

Conversely, define a functor $T_n: \overline{\mathcal{H}}^{i,j}(Vec_\Bbbk)^{(\ddot{H},{\ddot{A}}^{\ast cop})}(\ddot{\varphi}(n)) \rightarrow \mathcal{M}(\Phi)^H_A(n)$ as follows:

$\bullet$ for any object $(U,\a_U) \in \overline{\mathcal{H}}^{i,j}(Vec_\Bbbk)^{(\ddot{H},{\ddot{A}}^{\ast cop})}(\ddot{\varphi}(n))$ (suppose that the $\ddot{H}$-coaction is given by $u\mapsto u_{(0)} \o u_{(1)}$, and the ${\ddot{A}}^{\ast cop}$-coaction is given by $u \mapsto u_{[0]} \o u_{[1]}$),
$T_n(U):=U$ as a $\k$-module, and the $H$-Hom-comodule structure is also defined by $u\mapsto u_{(0)} \o u_{(1)}$, the $A$-Hom-module structure is defined by
$$
u \c a\mapsto u_{[1]}(a) u_{[0]}, ~~~~\mbox{where}~u \in U, ~a \in A;
$$

$\bullet$ for any morphism $\varpi: (U,\a_U)\rightarrow (V,\a_V)$, $R_n(\varpi):=\varpi$.

It is a straightforward computation to check that $T_n$ is well defined and $T_n$ is the inverse of $R_n$. This completes the proof.
\end{proof}

\begin{corollary}
If $\Phi:H \o A\rightarrow A \o H$ is a Hom-entwining map, then ${A^\ast}^{cop} \o H$ has a Hom-coalgebra structure over $\k$. Further, for any $n \in \mathbb{Z}$, $\mathcal{M}(\Phi)^H_A(n)$ is isomorphic to the category $Corep({A^\ast}^{cop} \o H)$.
\end{corollary}

\begin{proof}
Since Theorem 5.1, a Hom-entwining map $\Phi$ is in correspondence with a Hom-cotwistor $\varphi$. Then from Theorem 4.6, there is a Hom-coalgebra structure
over ${A^\ast}^{cop} \o H$ which is induced by $\varphi$.

Moreover, combining Proposition 4.7 and Proposition 5.5, we immediately get that $\mathcal{M}(\Phi)^H_A(n) \cong Corep({A^\ast}^{cop} \o H)$.
\end{proof}

\begin{definition}
Suppose that $H$, $A$ are all Hom-bialgebas over $\k$, and $\Phi:H \o A\rightarrow A \o H$ is a Hom-entwining map.
If $\Phi$ satisfies\\
$\left\{\begin{array}{l}
(E5)~ \sum {a_{1}}_\Phi \o {a_{2}}_\Psi \o {\a_H^2(h)}^{\Phi}{\a_H^2(g)}^\Psi = \sum {a_\Phi}_1 \o {a_\Phi}_2 \o \a_H^2({(hg)}^\Phi) ;\\
(E6)~ \sum \v_A(a_\Phi) {(1_H)}^\Phi = \v_A(a)1_H,
\end{array}\right.$\\
for any $h,g \in H$, $a \in A$, then the triple $(H,A,\Phi)$ is called a \emph{Hom-monoidal entwining datum}.
\end{definition}

Note that if $(A,\a_A)$ is a Hom-bialgebra over $\k$, then $(A^{\ast cop}, \a_{A^{\ast cop}})$ is also a Hom-bialgebra under the following structures:
\begin{eqnarray*}
&1_{A^{\ast cop}}:=\v,~~(f\ast g)(y):=f(\a_A^{-2}(y_1))g(\a_A^{-2}(y_2)),\\
&\overline{\v}(f):=f(1_A),~~\overline{\Delta}(f)(xy):=f(\a_A^{-2}(y x)),~~\a_{A^{\ast cop}}(f):=f\ci \a_A^{-1},
\end{eqnarray*}
where $x,y \in A$, $f \in A^{\ast cop}$.

\begin{proposition}
Suppose that $H$, $A$ are all Hom-bialgebas. Then there is a $\k$-linear map $\Phi: H\o A\rightarrow A \o H$ such that $(H,A,\Phi)$ is a Hom-monoidal entwining datum if and only if there is a
$\k$-linear map $\varphi: {A^\ast}^{cop} \o H\rightarrow H \o {A^\ast}^{cop}$ such that $\varphi$ is a monoidal Hom-cotwistor.
\end{proposition}

\begin{proof}
Recall from Theorem 5.1 that a Hom-entwining map $\Phi: H\o A\rightarrow A \o H$ is equivalent to a Hom-cotwistor $\varphi: {A^\ast}^{cop} \o H\rightarrow H \o {A^\ast}^{cop}$. We only need to check that $\Phi$ satisfies Eqs.(E5) and (E6) if and only if $\varphi$ satisfies Eqs.(M5) and (M6).

If $\Phi$ satisfies Eq.(E5), then we have
\begin{eqnarray*}
\sum (hg)^\varphi \o (f \ast f')^\varphi(a) &=& \sum (hg)^\Phi \o (f \ast f')(a_\Phi) \\
&=& \sum (hg)^\Phi \o f(\a_A^{-2}({a_\Phi}_1))f'(\a_A^{-2}({a_\Phi}_2)) \\
&\stackrel {E5}{=}& \sum \a_H^{-2}({\a_H^2(h)}^\Phi {\a_H^2(g)}^\Psi) \o f( \a_A^{-2}({a_1}_\Phi) )f'(\a_A^{-2}({a_2}_\Psi)) \\
&=& \sum h^\Phi g^\Psi \o f({\a_A^{-2}(a_1)}_\Phi)f'({\a_A^{-2}(a_2)}_\Psi)\\
&=& \sum h^\varphi g^\phi \o (f^\varphi f'^\phi)(a)
\end{eqnarray*}
for any $h,g \in H$, $f,f' \in A^\ast$, $a \in A$. Thus Eq.(M5) holds.

Conversely, one can show Eq.(E5) from (M5).

Similarly, Eq.(E6) and (M6) can be induced by each other.
\end{proof}

Define the following
monoidal structure in $\mathcal{M}(\Phi)^H_A(n)$:

$\bullet$ the associativity constraint and the unit constraint in $\mathcal{M}(\Phi)^H_A(n)$ is same as $\overline{\mathcal{H}}^{i,j}(Vec_\Bbbk)$;

$\bullet$ for any $(U, \a_U)$, $(V, \a_V) \in \mathcal{M}(\Phi)^H_A(n)$, the $A$-action and $H$-coaction on $(U \o V, \a_U \o \a_V)$ is given by
\begin{eqnarray*}
&(u \o v) \c a := u\c \a_A^{-i-2}(a_1) \o v \c \a_A^{-j-2}(a_2),\\
&(u \o v)_{(0)} \o (u \o v)_{(1)} := ( u_{(0)} \o v_{(0)} ) \o \a_H^i(u_{(1)}) \a_H^j(v_{(1)});
\end{eqnarray*}

$\bullet$ the tensor product of two arrows $f,g \in \mathcal{M}(\Phi)^H_A(n)$ is given by the
tensor product of $\k$-linear morphisms.

\begin{theorem}
For any $n \in \mathbb{Z}$, $(\mathcal{M}(\Phi)^H_A(n), \o, \k, a,l,r)$ is a monoidal category under the above structures if and only if $(H,A,\Phi)$ is a Hom-monoidal entwining datum.
\end{theorem}

\begin{proof}
$\Rightarrow$: Assume that $\mathcal{M}(\Phi)^H_A(n)$ is a monoidal category.
Since $(\k,id_\k) \in \mathcal{M}(\Phi)^H_A(n)$, one gets Eq.(E6).

Let us to show Eq.(E5).
Since Example 5.4(1), $H \o A$ is an object in $\mathcal{M}(\Phi)^H_A(n)$.
Then $(H \o A) \o (H \o A)$ is also an object in $\mathcal{M}(\Phi)^H_A(n)$, which implies
\begin{eqnarray*}
&&~~~~\rho^{(H\o A) \o (H \o A)}(( (\a_H^{-n-i-2}(h) \o 1_A) \o (\a_H^{-n-j-2}(g) \o 1_A) )\c a) \\
&&=\sum ((\a_H^{-n-i-2}(h) \o 1_A) \o (\a_H^{-n-j-2}(g) \o 1_A))_{(0)} \c \a_A(a_\Phi) \\
&&~~~~\o \a_H^{-n}( { \a_H^{n+1}(((\a_H^{-n-i-2}(h) \o 1_A) \o (\a_H^{-n-j-2}(g) \o 1_A))_{(1)}) }^\Phi )
\end{eqnarray*}
for any $h,g \in H$, $a \in A$.

Note that
\begin{eqnarray*}
&&~~~~\rho^{(H\o A) \o (H \o A)}(( (\a_H^{-n-i-2}(h) \o 1_A) \o (\a_H^{-n-j-2}(g) \o 1_A) )\c a) \\
&&= \sum \a_H^{-n-i-1}(h_1) \o \a_A^{n+1}({\a_A^{-n-i-2}(a_1)}_\Phi) \o \a_H^{-n-j-1}(g_1) \o \a_A^{n+1}({\a_A^{-n-j-2}(a_2)}_\Psi)\\
&&~~~~ \o \a_H^i({\a_H^{-n-i-1}(h_2)}^\Phi) \a_H^j({\a_H^{-n-j-1}(g_2)}^\Psi)\\
&&= \sum \a_H^{-n-i-1}(h_1) \o \a_A^{-i-1}({a_1}_\Phi) \o \a_H^{-n-j-1}(g_1) \o \a_A^{-j-1}({a_2}_\Psi) \o \a_H^{-n-2}({\a_H(h_2)}^\Phi) \a_H^{-n-2}({\a_H(g_2)}^\Psi),
\end{eqnarray*}
and
\begin{eqnarray*}
&&~~~\sum ((\a_H^{-n-i-2}(h) \o 1_A) \o (\a_H^{-n-j-2}(g) \o 1_A))_{(0)} \c \a_A(a_\Phi) \\
&&~~~~~\o \a_H^{-n}( { \a_H^{n+1}(((\a_H^{-n-i-2}(h) \o 1_A) \o (\a_H^{-n-j-2}(g) \o 1_A))_{(1)}) }^\Phi )\\
&&= \sum \a_H^{-n-i-1}(h_1) \o \a_A^{-i-1}({a_\Phi}_1) \o \a_H^{-n-j-1}(g_1) \o \a_A^{-j-1}({a_\Phi}_2) \o \a_H^{-n}((\a_H^{-1}(h) \a_H^{-1}(h))^\Phi),
\end{eqnarray*}
then consequently,
\begin{eqnarray*}
&&\sum \a_H^{-n-i-1}(h_1) \o \a_A^{-i-1}({a_1}_\Phi) \o \a_H^{-n-j-1}(g_1) \o \a_A^{-j-1}({a_2}_\Psi) \o \a_H^{-n-2}({\a_H(h_2)}^\Phi) \a_H^{-n-2}({\a_H(g_2)}^\Psi)\\
&&= \sum \a_H^{-n-i-1}(h_1) \o \a_A^{-i-1}({a_\Phi}_1) \o \a_H^{-n-j-1}(g_1) \o \a_A^{-j-1}({a_\Phi}_2) \o \a_H^{-n}((\a_H^{-1}(h) \a_H^{-1}(h))^\Phi).
\end{eqnarray*}
Applying $\v_H \o \a_A^{i+1} \o \v_H  \o \a_A^{j+1} \o \a_H^{n+2}$ on the both sides of the above equation, one gets
Eq.(E5).

$\Leftarrow$: Straightforward.
\end{proof}

\begin{corollary}
If $(H,A,\Phi)$ is a Hom-monoidal entwining datum, then for any $f,f' \in {A^\ast}$, $x,y,h\in H$, ${A^{\ast}}^{cop} \o H$ admits the following Hom-bialgeba structure:
\begin{eqnarray*}
&\widehat{\D}(f\o h):=\sum (f_2 \o {h_1}^\varphi) \o ({f_1}^\varphi \o h_2),~~~\widehat{\v}(f \o h):=\v_H(h)f(1_A),\\
&\a_{{A^\ast}^{cop} \o H}=\a_{A^\ast} \o \a_H,~~~(f \o x)(f' \o y):=f\ast f' \o xy,~~~1_{{A^\ast}^{cop}\otimes H}:=\v_H \o 1_A,
\end{eqnarray*}
where $\varphi: {A^\ast}^{cop} \o H\rightarrow H \o {A^\ast}^{cop}$ is induced by $\Phi$.

Further,
for any $n \in \mathbb{Z}$, $\mathcal{M}(\Phi)^H_A(n)$ is monoidal isomorphic to the category $Corep^{i,j}({A^\ast}^{cop} \o H)$.
Or equivalently,
$\mathcal{M}(\Phi)^H_A(n)$ and $\overline{\mathcal{H}}^{i,j}(Vec_\Bbbk)^{(\ddot{H},{\ddot{A}}^{\ast cop})}(\ddot{\varphi}(n))$ are monoidal isomorphic.
\end{corollary}

\begin{proof}
Straightforward from Proposition 5.8 and Theorem 4.10.
\end{proof}

\section{\textbf{Applications}}

\vskip 0.3cm
\subsection{\textbf{Hom-Doi-Hopf modules}}
\def\theequation{6.1.\arabic{equation}}
\setcounter{equation} {0} \hskip\parindent
\vskip 0.4cm

Let $(H,\a_H)$ be a Hom-bialgebra. A Hom-algebra $(A,\a_A)$ is called a \emph{right $H$-comodule algebra} if $A$ is a right $H$-comodule via $\r$, and
$$(ab)_{(0)}\o(ab)_{(1)}=a_{(0)}b_{(0)} \o a_{(1)}b_{(1)},~~~(1_A)_{(0)}\o(1_A)_{(1)}=1_A\o1_H,$$
for all $a,b\in A$. A Hom-coalgebra $(C,\a_C)$ is called a \emph{right $H$-module coalgebra} if $C$ is a right $H$-module and
$$\Delta_C( c\cdot h)=c_1\cdot h_1 \o c_2\cdot h_2,~~~\varepsilon_C(c\cdot h)=\varepsilon_C(c)\varepsilon_H(h),$$
for all $h\in H,c\in C$.

Suppose that $(A,\a_A)$ is a right $H$-comodule algebra and $(C,\a_C)$ is a right $H$-module coalgebra. For any integer $k \in \mathbb{Z}$,
an object $(U,\a_U)$ is called a \emph{$k$-th Doi-Hopf module over $(H,A,C)$} if $U$ is a right $A$-module and a right $C$-comodule satisfying
$$(u\cdot a)_{[0]}\o(u\cdot a)_{[1]}=u_{[0]}\cdot a_{(0)}\o u_{[1]}\cdot \a_H^k(a_{(1)}),$$
for all $a\in A$ and $m\in M$. Moreover, we call $(H,A,C)$ the \emph{Hom-Doi-Hopf datum}.

The category of $k$-th Doi-Hopf modules over $(H,A,C)$ and $A$-linear $C$-colinear homomorphisms will be denoted by $\mathcal{H}(k)_A^C$.

Let $m$ be an integer in $\mathbb{Z}$. If we define
\begin{equation*}
\begin{split}
\Phi: C\otimes A  &\longrightarrow A\otimes C \\
c\otimes a &\longmapsto \a_A^{-1}(a_{(0)}) \otimes \a_C^{-1}(c)\c \a_H^{m}(a_{(1)}),
\end{split}
\end{equation*}
where $c \in C$, $a \in A$, then it is a direct computation to check that $(C,A,\Phi)$ is a Hom-entwining structure.
Further, for $n \in \mathbb{Z}$, $(U,\a_U)$ is an object in $\mathcal{M}(\Phi)^C_A(n)$ means that
$$
\rho(u \c a) = u_{[0]}\c a_{(0)} \o u_{[1]} \c \a_H^{m-n}(a_{(1)}),~~u \in U,~~a \in A,
$$
which implies the category of $n$-th Hom-entwined modules is actually the category of $(m-n)$-th Doi-Hopf modules over $(H,A,C)$, i.e.,
$\mathcal{M}(\Phi)^H_H(n) = \mathcal{H}(m-n)_A^C$.

Recall from Theorem 5.1, for any $f \in A^\ast$, $c \in C$, the following Hom-cotwistor:
\begin{equation*}
\begin{split}
\varphi: {A^\ast}^{cop}\otimes C  &\longrightarrow C\otimes {A^\ast}^{cop} \\
f\otimes c &\longmapsto \sum f(\a_A^{-1}({e_i}_{(0)})) \a_C^{-1}(c) \c \a_A^{m}({e_i}_{(1)}) \o e^i,
\end{split}
\end{equation*}
is deduced from $\Phi$ which is defined above,
where $\sum e_i \o e^i$ are the dual bases of $A$ and ${A}^{\ast}$. Hence ${A^\ast}^{cop}\otimes C$ is a Hom-coalgebra under the following structures:
\begin{eqnarray*}
&\widehat{\D}(f\o c):=\sum f_1( \a_A^{-1}({e_i}_{(0)}) )f_2 \o \a_C^{-1}(c_1) \c \a_A^{m}({e_i}_{(1)}) \o (e^i \o c_2),\\
&\a_{{H^\ast}^{cop} \o H}=\a_{H^\ast} \o \a_H,~~\widehat{\v}(f \o h):=\v_H(h)f(1_H),
\end{eqnarray*}
where $f \in {A^\ast}$, $c\in C$. We call this Hom-coalgebra ${A^\ast}^{cop}\otimes C$ \emph{the $m$-th Doi-codouble of $C$ and $A$}.

Since Proposition 5.8 and Theorem 5.9, we have the following property.

\begin{theorem}
Supposed that $(H,A,C)$ is a Hom-Doi-Hopf datum, $(A,\a_A)$ and $(C,\a_C)$ are all Hom-bialgebras.
Then the following statements are equivalent:

(1) for any $c,d \in C$, $a \in A$, $\Phi$ satisfies
\begin{eqnarray*}
&{a_1}_{(0)} \o {a_2}_{(0)} \o ( c \c {a_1}_{(1)} ) ( d \c {a_2}_{(1)} ) = {a_{(0)}}_1 \o {a_{(0)}}_2 \o (cd) \c \a_H^2(a_{(1)}), \\
&\v_A(a)1_C = 1_C \c a_{(1)} \v_A(a_{(0)});
\end{eqnarray*}

(2) $(C,A,\Phi)$ is a Hom-monoidal entwining structure;

(3) $\varphi$ is a monoidal Hom-cotwistor;

(4) the $m$-th Doi-codouble ${A^\ast}^{cop}\otimes C$ is a Hom-bialgebra under the following Hom-algebra structure:
\begin{eqnarray*}
(f \o x)(f' \o y):=f\ast f' \o xy,~~~~~1_{{A^\ast}^{cop}\otimes C}:=\v_A \o 1_C,
\end{eqnarray*}
where $f, f'\in A^\ast$, $x,y \in C$;

(5) $\mathcal{H}(k)_A^C$ is a monoidal category with the following structures:

$\bullet$ the associativity constraint and the unit constraint in $\mathcal{H}(k)_A^C$ is same as $\overline{\mathcal{H}}^{i,j}(Vec_\Bbbk)$;

$\bullet$ for any $(U,\a_U), (V, \a_V) \in \mathcal{H}(k)_A^C$, the Hom-module and Hom-comodule structure on $U \o V$ are given by
\begin{eqnarray*}
(u \o v)\c a:= u \c \a_A^{-i-2}(a_1) \o v \c \a_A^{-j-2}(a_2),\\
\rho^{U \o V}(u \o v):=u_{[0]} \o v_{[0]} \o \a_C^i(u_{[1]})\a_C^j(v_{[1]});
\end{eqnarray*}

$\bullet$ the tensor product of two arrows $f,g \in \mathcal{H}(k)_A^C$ is given by the
tensor product of $\k$-linear morphisms.
Further, $\mathcal{H}(k)_A^C$ and $Corep^{i,j}({A^\ast}^{cop}\otimes C)$ are isomorphic as monoidal categories.
\end{theorem}

\vskip 0.3cm
\subsection{\textbf{Hom-Long dimodules and Hom-$\mathcal{D}$-equation}}
\def\theequation{6.2.\arabic{equation}}
\setcounter{equation} {0} \hskip\parindent
\vskip 0.4cm

Suppose that $(H,\a_H)$ is a Hom-bialgebra over $\k$.
If $(U,\a_U)$ is both a right $H$-Hom-module and a right $H$-Hom-comodule, and satisfies the following compatibility condition:
\begin{eqnarray*}
\rho(u \c h) = u_{(0)}\c \a_H(h)\o \a_H(u_{(1)}),
\end{eqnarray*}
for all $u \in U$ and $h \in H$, then we call $(U,\a_U)$ a \emph{Hom-Long dimodule}.
We denote by $\mathcal{L}^H_H$ the category of Hom-Long dimodules, morphisms being $H$-linear and $H$-colinear (see \cite{fl} for more detail of the dimodules).

\begin{proposition}
$\mathcal{L}^H_H$ is a monoidal category with the following structures:

$\bullet$ the associativity constraint and the unit constraint in $\mathcal{L}^H_H$ is same as $\overline{\mathcal{H}}^{i,j}(Vec_\Bbbk)$;

$\bullet$ for any $(U,\a_U), (V, \a_V) \in \mathcal{L}^H_H$, the Hom-module and Hom-comodule structure on $U \o V$ are given by
\begin{eqnarray*}
(u \o v)\c h:= u \c \a_H^{-i-2}(h_1) \o v \c \a_H^{-j-2}(h_2),\\
\rho^{U \o V}(u \o v):=u_{(0)} \o v_{(0)} \o \a_H^i(u_{(1)})\a_H^j(v_{(1)});
\end{eqnarray*}

$\bullet$ the tensor product of two arrows $f,g \in \mathcal{L}^H_H$ is given by the
tensor product of $\k$-linear morphisms.
\end{proposition}

\begin{proof}
Straightforward.
\end{proof}

If we define $\Phi: H\otimes H \rightarrow H\otimes H$ by $\Phi(c\otimes a):=a\otimes c$ for any $c,a \in H$,
then it is a direct computation to check that $(H,H,\Phi)$ is a Hom-monoidal entwining datum. Hence for any $n \in \mathbb{Z}$, $\mathcal{M}(\Phi)^H_H(n)$ is a monoidal category.
Further, $(U,\a_U)$ is an object in $\mathcal{M}(\Phi)^H_H(n)$ means that
$$
\rho(u \c a) = u_{(0)}\c \a_H(a) \o \a_H(u_{(1)}),~~u \in U,~~a \in H,
$$
which implies $\mathcal{M}(\Phi)^H_H(n) = \mathcal{L}^H_H$.

Recall from Proposition 5.8, $\Phi$ could induce the following monoidal Hom-cotwistor:
\begin{equation*}
\begin{split}
\varphi: {H^\ast}^{cop}\otimes H  &\longrightarrow H\otimes {H^\ast}^{cop} \\
f\otimes h &\longmapsto h \o f,
\end{split}
\end{equation*}
where $f \in H^\ast$, $h \in H$. Obviously ${H^\ast}^{cop}\otimes H$ is a Hom-bialgebra under the usual tensor structures.
We call this Hom-bialgebra ${H^\ast}^{cop}\otimes H$ the \emph{Hom-Long-codouble of $H$}.
Furthermore, since Corollary 5.10, the category of Hom-Long dimodules is isomorphic to the corepresentations of ${H^\ast}^{cop}\otimes H$
as monoidal categories.

Recall from \cite{gm} that we have the following definition.

\begin{definition}
Let $U$ be a vector space over $\k$ and $R \in End_\k(U \o U)$. We say that $R$ is a \emph{solution of the $\mathcal{D}$-equation} if
$$
R^{12}R^{23} = R^{23}R^{12}
$$
in $End_\k(U \o U \o U)$.
\end{definition}

\begin{theorem}
Let $(H,\a_H)$ be a Hom-bialgebra over $\k$, $\mathcal{L}^H_H$ denote the category of Hom-Long dimodules of $H$.
For all integer $m \in \mathbb{Z}$, if we define the following $\k$-linear map
\begin{equation*}
\begin{split}
\xi_{U,V}: U\otimes V  &\longrightarrow U\otimes V \\
u\otimes v &\longmapsto \a_U^{-1}(u) \c \a_H^m(v_{(1)}) \otimes \a_V^{-1}(v_{(0)}),
\end{split}
\end{equation*}
where $(U,\a_U), (V,\a_V) \in \mathcal{L}^H_H$, then $\xi$ satisfies the following generalized Hom-type $\mathcal{D}$-equation in $\overline{\mathcal{H}}^{i,j}(Vec_\Bbbk)$:
$$\aligned
\xymatrix{
(U\otimes V)\otimes W \ar[d]|{\xi_{U,V} \o id_W} \ar[rr]^{a_{U,V,W}} & & U\otimes (V\otimes W) \ar[rr]^{id_{U}\o \xi_{V, W}}&& U\otimes (V\otimes W) \ar[rr]^{a^{-1}_{U,V, W}} & & (U\otimes V)\otimes W \ar[d]|{\xi_{U,V} \o id_W} \\
(U\o V) \o W \ar[rr]_{a_{U,W, V}} & & U \o (V\o W) \ar[rr]_{id_U \o \xi_{V,W}} && U \o (V\o W) \ar[rr]_{a^{-1}_{U,V, W}} && (U\otimes V)\otimes W.}
\endaligned$$
\end{theorem}

\begin{proof}
For any $u \in U$, $v \in V$, $w \in W$, since the following identities
\begin{eqnarray*}
&&~~~( (\xi_{U,V} \o id_W) \ci a^{-1}_{U,W, V} \ci (id_{U}\o \xi_{V, W}) \ci a_{U,V,W} ) ((u\otimes u)\otimes w) \\
&&= ( (\xi_{U,V} \o id_W) \ci a^{-1}_{U,W, V} ) ( \a_U^{i+1}(u) \o (\a_V^{-1}(v) \c \a_H^{m-j-1}(w_{(1)}) \o \a_W^{-j-2}(w_{(0)})) ) \\
&&= (\a_U^{-1}(u) \c \a_H^{m}(v_{(1)}) \otimes \a_V^{-2}(v_{(0)}) \c \a_H^{m-j-1}(w_{(1)}))\otimes \a_W^{-1}(w_{(0)}) \\
&&= a^{-1}_{U,V, W}( \a_U^{i}(u)\c \a_H^{m+i+1}(v_{(1)}) \o
(\a_V^{-2}(v_{(0)}) \c \a_H^{m-j-1}(w_{(1)}) \o \a_W^{-j-2}(w_{(0)}) ) ) \\
&&= (a^{-1}_{U,V, W} \ci(id_U \o \xi_{V,W})) (\a_U^{i}(u)\c \a_H^{m+i+1}(v_{(1)}) \o ( \a_V^{-1}(v_{(0)}) \o \a_W^{-j-1}(w) )) \\
&&= (a^{-1}_{U,V, W} \ci (id_U \o \xi_{V,W}) \ci a_{U,W, V} \ci (\xi_{U,V} \o id_W) ) ((u\otimes u)\otimes w),
\end{eqnarray*}
the conclusion holds.
\end{proof}

\begin{corollary}
For all integer $q \in \mathbb{Z}$, if we define the linear form $\zeta \in ( ({H^\ast}^{cop}\o H) \o ({H^\ast}^{cop}\o H) )^\ast$ by
$$
\zeta(f\o x,f' \o y):= f(\a_H^{q}(y)) \v_H(x)f'(1_H),
$$
then $\zeta$ satisfies the following $\mathcal{D}$-type equation
$$
\zeta^{12}\ast \zeta^{23} = \zeta^{23}\ast \zeta^{12},
$$
where $\zeta^{12}(a,b,c) = \zeta(a,b)\v_{{H^\ast}^{cop}\o H}(c)$, $\zeta^{23}(a,b,c) = \v_{{H^\ast}^{cop}\o H}(a)\zeta(b,c)$ for any $a,b,c \in {H^\ast}^{cop}\o H$.
\end{corollary}

\begin{proof}
For any $f \o x, f' \o y, f'' \o z \in {H^\ast}^{cop}\o H$, we compute
\begin{eqnarray*}
&&~~~(\zeta^{12}\ast \zeta^{23}) (f \o x, f' \o y, f'' \o z) \\
&&=\zeta^{12} (\a_{H^\ast}^{-2}(f_2) \o \a_H^{-2}(x_1), \a_{H^\ast}^{-2}(f'_2) \o \a_H^{-2}(y_1), \a_{H^\ast}^{-2}(f''_2) \o \a_H^{-2}(z_1) )\\
&&~~~~~~ \zeta^{23}(\a_{H^\ast}^{-2}(f_1) \o \a_H^{-2}(x_2), \a_{H^\ast}^{-2}(f'_1) \o \a_H^{-2}(y_2), \a_{H^\ast}^{-2}(f''_1) \o \a_H^{-2}(z_2) )\\
&&= \a_{H^\ast}^{-2}(f_2) ( \a_H^{q-2}(y_1) ) \v_H(x_1) f'_2( 1_H ) \v_H(z_1) f''_2( 1_H )
\a_{H^\ast}^{-2}(f'_1) ( \a_H^{q-2}(z_2) ) \v_H(y_2) f''_1( 1_H ) \v_H(x_2) f_1( 1_H )\\
&&= f(\a_H^{q}(y)) \v_H(x) f'(\a_H^{q}(z)) f''( 1_H ),
\end{eqnarray*}
and
\begin{eqnarray*}
&&~~~(\zeta^{23}\ast \zeta^{12}) (f \o x, f' \o y, f'' \o z) \\
&&=\zeta^{23} (\a_{H^\ast}^{-2}(f_2) \o \a_H^{-2}(x_1), \a_{H^\ast}^{-2}(f'_2) \o \a_H^{-2}(y_1), \a_{H^\ast}^{-2}(f''_2) \o \a_H^{-2}(z_1) )\\
&&~~~~~~ \zeta^{12}(\a_{H^\ast}^{-2}(f_1) \o \a_H^{-2}(x_2), \a_{H^\ast}^{-2}(f'_1) \o \a_H^{-2}(y_2), \a_{H^\ast}^{-2}(f''_1) \o \a_H^{-2}(z_2) )\\
&&= \v_H(x_1) f_2( 1_H ) \a_{H^\ast}^{-2}(f'_2) ( \a_H^{q-2}(z_1) ) \v_H(y_1) f''_2( 1_H )
\a_{H^\ast}^{-2}(f_1) ( \a_H^{q-2}(y_2) ) \v_H(x_2) f'_1( 1_H )\v_H(z_2) f''_1( 1_H )\\
&&= f(\a_H^{q}(y)) \v_H(x) f'(\a_H^{q}(z)) f''( 1_H ),
\end{eqnarray*}
thus the conclusion holds.
\end{proof}

\vskip 0.3cm
\subsection{\textbf{Hom-Yetter-Drinfeld modules and Hom-Yang-Baxter equation}}
\def\theequation{6.3.\arabic{equation}}
\setcounter{equation} {0} \hskip\parindent
\vskip 0.4cm

Let $(H,\a_H)$ be a finite dimensional Hom-Hopf algebra over $\k$, $p$ be an intege in $\mathbb{Z}$.
Recall from [\cite{CoDr}, Proposition 3.3] that if $(U,\a_U)$ is both a right $H$-Hom-module and a right $H$-Hom-comodule, and satisfies the following compatibility condition:
\begin{eqnarray}
\rho(u \c h) = u_{(0)}\c \a_H^{-1}(h_{(21)})\o S(\a_H^{p-2}(h_1))(\a_H^{-1}(u_{(1)}) \a_H^{p-4}(h_{22})),
\end{eqnarray}
for all $u \in U$ and $h \in H$, then we call $(U,\a_U)$ a \emph{$p$-th right-right Yetter-Drinfeld module of $H$}.
We denote by $\mathcal{YD}^H_H(p)$ the category of $p$-th right-right Yetter-Drinfeld modules, morphisms being $H$-linear and $H$-colinear.

Note that $\mathcal{YD}^H_H(p)$ is a braided monoidal category with the following structures (see [\cite{CoDr}, Proposition 3.5 and Theorem 3.6]):

$\bullet$ the associativity constraint and the unit constraint in $\mathcal{YD}^H_H(p)$ is same as $\overline{\mathcal{H}}^{i,j}(Vec_\Bbbk)$;

$\bullet$ for any $(U,\a_U), (V, \a_V) \in \mathcal{YD}^H_H(p)$, the Hom-module and Hom-comodule structure on $U \o V$ are given by
\begin{eqnarray*}
(u \o v)\c h:= u \c \a^{-i-2}(h_1) \o v \c \a^{-j-2}(h_2),\\
\rho^{U \o V}(u \o v):=u_{(0)} \o v_{(0)} \o \a^i(u_{(1)})\a^j(v_{(1)});
\end{eqnarray*}

$\bullet$ the tensor product of two arrows $f,g \in \mathcal{YD}^H_H(p)$ is given by the
tensor product of $\k$-linear morphisms;

$\bullet$ for any $U,V \in \mathcal{YD}^H_H(p)$, $u \in U$, $v \in V$, the braiding is given by
\begin{eqnarray*}
\t_{U,V}:U \o V\rightarrow V \o U, ~~~~u \o v\mapsto \a_V^{j-i-1}(v_{(0)}) \o \a_U^{i-j-1}(u) \c \a^{-p}(v_{(1)}).
\end{eqnarray*}

Let $m$ be any integer in $\mathbb{Z}$. If we define
\begin{equation*}
\begin{split}
\Phi: H\otimes H  &\longrightarrow H\otimes H \\
c\otimes a &\longmapsto \a_H^{-2}(a_{21})\otimes S(\a_H^{m-2}(a_1) (\a_H^{-2}(c) \a_H^{m-4}(a_{22})),
\end{split}
\end{equation*}
where $c,a \in H$, then it is a direct computation to check that $(H,H,\Phi)$ is a Hom-monoidal entwining datum, thus $\mathcal{M}(\Phi)^H_H(n)$ is a monoidal category.
Further, for any $n \in \mathbb{Z}$, $(U,\a_U)$ is an object in $\mathcal{M}(\Phi)^H_H(n)$ means that
$$
\rho(u \c h) = u_{(0)}\c \a^{-1}(h_{(21)})\o S(\a^{m-n-2}(h_1))(\a^{-1}(u_{(1)}) \a^{m-n-4}(h_{22})),~~u \in U,~~h \in H,
$$
which implies $\mathcal{M}(\Phi)^H_H(n) = \mathcal{YD}^H_H(m-n)$.

Recall from Proposition 5.8, $\Phi$ could induce the following monoidal Hom-cotwistor:
\begin{equation*}
\begin{split}
\varphi: {H^\ast}^{cop}\otimes H  &\longrightarrow H\otimes {H^\ast}^{cop} \\
f\otimes h &\longmapsto \sum f(\a_H^{-2}({e_i}_{21})) S(\a_H^{m-2}({e_i}_{1})) ( \a_H^{-2}(h) \a_H^{m-4}({e_i}_{22}) ) \o e^i,
\end{split}
\end{equation*}
where $f \in H^\ast$, $h \in H$, $\sum e_i \o e^i$ are the dual bases of $H$ and ${H}^{\ast}$. Hence ${H^\ast}^{cop}\otimes H$ is a Hom-bialgebra under the following structures:
\begin{eqnarray*}
&\widehat{\D}(f\o h):=\sum (f_1(\a_H^{-2}({e_i}_{21}))f_2 \o S(\a_H^{m-2}({e_i}_{1})) ( \a_H^{-2}(h_1) \a_H^{m-4}({e_i}_{22}) )) \o (e^i \o h_2),\\
&\a_{{H^\ast}^{cop} \o H}=\a_{H^\ast} \o \a_H,~~\widehat{\v}(f \o h):=\v_H(h)f(1_H),\\
&(f \o x)(f' \o y):=f\ast f' \o xy,~~~~~1_{{H^\ast}^{cop}\otimes H}:=\v_H \o 1_H,
\end{eqnarray*}
where $f,f' \in {H^\ast}$, $x,y,h\in H$. We call this Hom-bialgebra ${H^\ast}^{cop}\otimes H$ \emph{the $m$-th Dinfeld codouble of $H$} (see [\cite{CoDr}, Definition 5.2]).
Furthermore, since Corollary 5.10, for any integer $p$, the category of $p$-th Yetter-Drinfeld modules is isomorphic to the corepresentations of ${H^\ast}^{cop}\otimes H$ as monoidal categories.

\begin{corollary}
The family of maps $\t_{U,V}$ for any $U, V \in \mathcal{YD}^H_H(p)$ is a solution of the following Hom-Yang-Baxter equation:
\begin{eqnarray*}
&&(id_{W} \o \t_{U,V}) \circ a_{W, U, V} \circ (\t_{U,W} \o id_V)\circ
a^{-1}_{U,W,V}\circ (id_U \o \t_{V,W}) \circ a_{U,V,W}\\
&&=a_{W,V,U}\circ (\t_{W,V} \o id_U)\circ a^{-1}_{W,V,U} \circ (id_{V} \o \t_{U,W})\circ a_{V,U,W}
\circ (\t_{U,V} \o id_W).
\end{eqnarray*}
\end{corollary}

\begin{proof}
Straightforward.
\end{proof}

\begin{proposition}
${H^\ast}^{cop}\otimes H$ has a coquasitriangular structure
\begin{eqnarray*}
\xi(f \o x, f' \o y):=f(\a^{-m}(y))\v(x)f'(1_H).
\end{eqnarray*}
where $f,g \in H^\ast$, $x,y\in H$. Furthermore, $Corep^{i,j}({H^\ast}^{cop}\otimes H)$ and $\mathcal{YD}^H_H(p)$ are isomorphic as braided monoidal
categories.
\end{proposition}

\begin{proof}
See [\cite{CoDr}, Definition 5.2 and Theorem 5.4].
\end{proof}

\

\begin{center}
 {\bf ACKNOWLEDGEMENT}
\end{center}
The work was partially supported by the NSF of China (NO. 11371088), and the NSF of Qufu Normal University (NO. xkj201514).
\\

 \end{document}